\documentclass[]{amsart}
\usepackage{amssymb}
\usepackage{amsthm}
\usepackage{url}
\usepackage{amssymb}
\usepackage{amsmath}
\usepackage{stmaryrd}
\usepackage{siunitx}
\usepackage{commath}
\usepackage{enumitem}
\usepackage{cleveref}
\usepackage{xcolor}
\usepackage{xspace,color}
\usepackage{caption,subcaption}
\usepackage{soul}
\usepackage{ragged2e}
\usepackage{tikz}
\usepackage{longtable}
\usepackage{xltabular}
\usepackage{supertabular}
\usepackage[normalem]{ulem}
\usepackage{accents}
\usepackage{amsaddr}

\usepackage{graphicx}
\usepackage{amsfonts, verbatim}

\makeatletter
\newcommand{\tnorm}{\@ifstar\@tnorms\@tnorm}
\newcommand{\@tnorm}[2][]{%
  \mathopen{#1|\mkern-1.5mu#1|\mkern-1.5mu#1|}
  #2
  \mathclose{#1|\mkern-1.5mu#1|\mkern-1.5mu#1|}
}
\makeatother
\newcommand{\jump}[1]{\llbracket #1 \rrbracket}

\newtheorem{theorem}{Theorem}
\newtheorem{lemma}{Lemma}
\newtheorem{corollary}{Corollary}

\newtheorem{remark}{Remark}
\begin{document} 
%---------------------------------------------------------------------
\author{Aycil Cesmelioglu$^1$, Jeonghun J. Lee$^2$, Sander Rhebergen$^3$, Dorisa Tabaku$^4$ }
\address{$^1$Department of Mathematics and Statistics, Oakland
  University, Michigan, USA \\ 
  $^2$ Department of Mathematics, Baylor University, Waco, Texas, USA \\
  $^3$ Department of Applied Mathematics, University of
  Waterloo, Ontario, Canada \\
  $^4$Department of Mathematics and Statistics, Oakland
  University, Michigan, USA}

\email{$^1$cesmelio@oakland.edu, $^2$jeonghun\_lee@baylor.edu}
\email{$^3$srheberg@uwaterloo.ca, $^4$dorisatabaku@oakland.edu}

%---------------------------------------------------------------------
  \title[An HDG method for the dual-porosity-Stokes problem]{A hybridizable discontinuous Galerkin method for the
    dual-porosity-Stokes problem}
    
%-------------------------------------------

%---------------------------------------------------------------------
  \begin{abstract}
    We introduce and analyze a hybridizable discontinuous Galerkin
    (HDG) method for the dual-porosity-Stokes problem. This coupled
    problem describes the interaction between free flow in
    macrofractures/conduits, governed by the Stokes equations, and
    flow in microfractures/matrix, governed by a dual-porosity
    model. We prove that the HDG method is strongly conservative,
    well-posed, and give an a priori error analysis showing dependence
    on the problem parameters. Our theoretical findings are
    corroborated by numerical examples.
  \end{abstract}
%---------------------------------------------------------------------
\date{August, 2023}
\maketitle

%---------------------------------------------------------------------
\section{Introduction}
\label{sec:intro}

The interaction between porous media flow and free flow arises from
different flow problems in engineering such as industrial filtration,
groundwater discharge, and petroleum and gas extraction (see, for
example,
\cite{Arbogast:2007,Arbogast:2009,Badea:2010,Cesmelioglu:2009,Discacciati:2004}). These
problems are typically modeled by the coupled Stokes--Darcy equations
and many numerical methods have been designed for this model, see for
example,
\cite{Cesmelioglu:2020,Kanschat:2010,RiviereYotov:2005,Layton:2003,Burman:2007,Mu:2007,Discacciati:2009,Cao:2010}
and references therein. However, the coupled Stokes--Darcy model does
not account for the heterogeneous nature of a porous medium, which in
practice may contain multiple porosities. To address this, Hou et
al. \cite{Hou:20162} introduce the dual-porosity-Stokes model. In
these equations, flow in microfractures and the matrix are modelled by
a dual-porosity model \cite{Warren:1963}, while mass conservation,
force balance, the Beavers--Joseph--Saffman condition
\cite{Beavers:1967,Saffman:1971}, and a no-exchange condition, are
imposed on the interface between the free flow and porous media flow
domains. The first of these three interface conditions also appear in the
coupled Stokes--Darcy model while the no-exchange condition is
specific to the dual-porosity-Stokes model; it prescribes that fluid
in the matrix can flow into the microfractures, but not into the
conduits directly.

A weak formulation of the time-dependent dual-porosity-Stokes model is
presented by Hou et al. \cite{Hou:20162}.  They furthermore show that
the model is well-posed and propose and analyze a monolithic finite
element method for the model. Al Mahbub et
al. \cite{Mahbub:2019b,Mahbub:2019} introduce and analyze stabilized
mixed finite element methods for the time-dependent and stationary
cases, respectively.  More recently, Wen et al. \cite{WEN:2022}
introduce and analyze a monolithic and strongly conservative scheme
for the stationary dual-porosity-Stokes problem based on symmetric
interior penalty discontinuous Galerkin and mixed finite
element methods, while Qiu et al. \cite{Qiu:2023} present and analyze a weak formulation for the stationary dual-porosity-Navier--Stokes
model under a small
data assumption. They also propose and
analyze a corresponding finite element method. Furthermore, for
time-dependent dual-porosity-Stokes models, various decoupled schemes
have been studied, see for example
\cite{Mahbub:2019b,Mahbub:2019,Shan:2019,Gao:2021,Cao:2021-1,Cao:2022,Li:2022}.

Well-posedness of the
weak formulation of the time-dependent dual-porosity-Stokes problem
was proven in \cite{Hou:2016} using a G\r{a}rding-type
inequality. In this paper we follow a different approach. We
consider the weak formulation presented in \cite{WEN:2022} for the
stationary problem in mixed form, and show well-posedness using
saddle point theory. We then propose a monolithic hybridizable
discontinuous Galerkin (HDG) method for the dual-porosity-Stokes
problem. This HDG method couples a pressure-robust IP-HDG method for
Stokes \cite{Rhebergen:2017} to a hybridized BDM discretization
\cite{Arnold:1985,Boffi:book,Du:book} of the dual-porosity
problem. Let us remark that, in the absence of source terms, our
discretization is strongly conservative, i.e., the velocity field is
pointwise divergence-free and divergence-conforming
\cite{Kanschat:2010}. Furthermore, for higher-order accurate
approximations, hybridizable DG/BDM methods typically have much less
globally coupled degrees-of-freedom compared to usual DG methods on
the same mesh \cite{Cockburn:2009a}.

The remainder of this paper is organized as follows. We introduce the
dual-porosity-Stokes model in \Cref{sec:model}, and prove
well-posedness of the weak formulation of this model in \Cref{sec:weak
  form}.  We present and show well-posedness of our strongly
conservative HDG method for the dual-porosity-Stokes problem in
\Cref{sec:hdg}. An a priori error analysis of the discretization,
showing explicit dependence on problem parameters, is presented in
\Cref{sec:error}, while numerical examples are presented in
\Cref{sec:Numerical ex}. We conclude in \Cref{sec:Conclusions}.

%---------------------------------------------------------------------
\section{The dual-porosity-Stokes model}
\label{sec:model}

Let $\Omega \subset \mathbb{R}^{\text{dim}}$, $\text{dim}=2,3$, be a
domain with Lipschitz boundary $\partial\Omega$ and let $\Omega^s$ be
a \emph{free flow} domain and $\Omega^d$ a \emph{porous medium}
domain. The free flow and porous medium domains are nonoverlapping,
i.e., $\Omega^s \cap \Omega^d = \emptyset$, and are such that
$\overline{\Omega} = \overline{\Omega}^s \cup
\overline{\Omega}^d$. Let
$\Gamma^I=\overline{\Omega}^s\cap \overline{\Omega}^d$,
$\Gamma^s=\partial\Omega^s\cap \partial\Omega$, and
$\Gamma^d=\partial\Omega^d\cap \partial\Omega$. We denote by $n^j$ the
outward pointing unit normal vector of $\Omega^j$ for $j=s,d$. The
unit normal vector on $\Gamma^I$ is denoted by $n$ and coincides with
$n^s = -n^d$. See \Cref{fig:domain} for an illustration of a two
dimensional domain $\Omega$.

\begin{figure}[tbp]
  \centering
  \begin{tikzpicture}[scale=1.8]
    \draw [color=black,fill = purple!30!pink](0,0) rectangle (2,1);
    \draw [thin,color=black] (0,1)--(2,1)
    node[ near start,above ]{$\Gamma^I$};
    \draw [ultra thin,color=black] (0,0)--(2,0)
    node[midway,below]{$\Gamma^d$};
    \draw [ultra thin,color=black] (0,2)--(2,2) node[midway,above]{$\Gamma^s$};
    \draw[->] (1.5,1)--(1.5,0.75) node[midway,below,right]{$n$};
    \draw [color=black] (0,1)--(0,2)--(2,2)--(2,1);
    \node at (1,0.5) {$\Omega^d$};
    \node at (1,1.5) {$\Omega^s$};
  \end{tikzpicture}
  \caption{Illustration of a free flow/porous medium domain $\Omega$
    in two dimensions.}
  \label{fig:domain}
\end{figure}
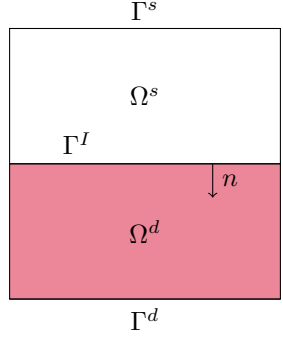

Given the kinematic viscosity $\mu$ and source term $f$, the free flow
fluid velocity $u$ and fluid pressure $p$ in $\Omega^s$ satisfy the
Stokes equations:
\begin{subequations}
  \label{eq:Stokes}
  \begin{align}
    \label{eq:Stokes-a}
    -\nabla\cdot(2\mu\epsilon(u))+\nabla p &=f && \text{ in } \Omega^s,
    \\ 
    \label{eq:Stokes-b}
    \nabla \cdot u & =0 && \text{ in } \Omega^s,
    \\ 
    \label{eq:Stokes-c}
    u & =0 && \text{ on } \Gamma^s,
  \end{align}
\end{subequations}
where $\epsilon(u)=(\nabla u+(\nabla u)^T)/2$ is the strain rate
tensor.

In $\Omega^d$, the matrix pressure $p^m$, matrix velocity $u^m$,
pressure in the microfractures $p$, and velocity in the microfractures
$u$ satisfy the dual-porosity model:
\begin{subequations}
  \label{eq:dualporosity}
  \begin{align}
    \label{eq:dualporosity-a}
    \kappa_f^{-1}u+\nabla p&=0  &&\text{ in } \Omega^d,
    \\
    \label{eq:dualporosity-b}
    \nabla \cdot u+\sigma\kappa_m(p-p^m)&=g  &&\text{ in } \Omega^d,
    \\ 
    \label{eq:dualporosity-c}
    \kappa_m^{-1}u^m+\nabla p^m&=0  &&\text{ in } \Omega^d,
    \\
    \label{eq:dualporosity-d}
    \nabla \cdot u^m+\sigma\kappa_m(p^m-p)&=0  &&\text{ in } \Omega^d,
    \\ 
    \label{eq:dualporosity-e}
    u^m\cdot n&=0  && \text{ on } \Gamma^d,
    \\ 
    \label{eq:dualporosity-f}
    u\cdot n&=0  && \text{ on } \Gamma^d, 
  \end{align}
\end{subequations}
where $\kappa_m$ and $\kappa_f$ are the intrinsic permeabilities in
the matrix and microfractures, respectively, and $g$ is a source
term. Furthermore, the shape factor $\sigma$ is a geometrical
parameter related to the morphology and dimension of the
microfractures that controls the fluid flow from the matrix to the
microfractures \cite{Warren:1963}. We assume that $0 < \sigma \leq \sigma^*$ for some constant $\sigma^*>0$.

The interface conditions, that couple the Stokes model and the
dual-porosity model, are given by:
\begin{subequations}
  \label{eq:IC}
  \begin{align}
    \label{eq:IC-a}
    u^m\cdot n&=0 &&\text{ on } \Gamma^I,
    \\
    \label{eq:IC-b}
    u^s\cdot n-u^d\cdot n&=0 &&\text{ on } \Gamma^{I},
    \\
    \label{eq:IC-c}
    -2\mu\epsilon(u^s)n\cdot n+p^s&= p^d&&\text{ on } \Gamma^{I}, 
    \\
    \label{eq:IC-d}
    -2\mu(\epsilon(u^s)n)^t &=\alpha \mu\kappa_f^{-1/2}(u^s)^t &&\text{ on } \Gamma^{I},
  \end{align}
\end{subequations}
where $u^j:=u|_{\Omega^j}$, $p^j:=p|_{\Omega^j}$ for $j=s,d$, $\alpha$
is a constant, and $w^t=w-(w\cdot n)n$ is the tangential component of
a vector $w$. The first interface condition \cref{eq:IC-a} describes
that there is no flow across the interface from the matrix to the
conduits. The remaining interface conditions describe the exchange
between the conduits/macrofractures and the microfractures and are
similar to those of the Stokes--Darcy model \cite{Layton:2003}.
Specifically, \cref{eq:IC-b} denotes mass conservation and
\cref{eq:IC-c} denotes the balance of forces between the
microfractures and the conduits, while \cref{eq:IC-d} is the
Beavers--Joseph--Saffman interface condition
\cite{Beavers:1967,Saffman:1971}.

%----------------------------------------
\section{The weak formulation}
\label{sec:weak form}

In this section, we present the weak formulation for
\cref{eq:Stokes,eq:dualporosity,eq:IC}. We denote the $L^2$-inner
product on a domain $E \subset \mathbb{R}^{\text{dim}}$ by
$(\cdot, \cdot)_E$, and on a $S \subset \mathbb{R}^{\text{dim}-1}$ by
$\langle \cdot, \cdot \rangle_S$. Furthermore, we define the following
standard Hilbert spaces:
\begin{equation*}
  \begin{split}
    H^1_{\Gamma^s}(\Omega^s)&:=\cbr[1]{v\in H^1(\Omega^s):v=0 \text{ on }\Gamma^s},
    \\
    H(\text{div},\Omega)&:=\cbr[1]{v \in [L^2(\Omega)]^{\text{dim}}: \nabla \cdot v \in L^2(\Omega)},
    \\
    H_0(\text{div},\Omega^d)&:=\cbr[1]{v \in H(\text{div},\Omega^d): v\cdot n=0 \text{ on }\partial\Omega^d},
    \\
    L_0^2(\Omega)&:=\cbr[1]{q\in L^2(\Omega): (q, 1)_{\Omega} = 0 },
    \\
    L_0^2(\Omega^d)&:=\cbr[1]{q\in L^2(\Omega^d): (q, 1)_{\Omega^d} = 0}.
  \end{split}
\end{equation*}
Recall that the space $H(\text{div}, \Omega)$ is equipped with the
norm
\begin{equation*}
  \norm[0]{v}_{H(\text{div},\Omega)} := ( \norm[0]{v}_{\Omega}^2 + \norm[0]{\nabla \cdot v}_{\Omega}^2)^{1/2}.
\end{equation*}
Let us next define the following function spaces:
\begin{align*}
  V &:=
      \cbr[1]{v\in H(\text{div},\Omega): v|_{\Omega^s} \in H^1_{\Gamma^s}(\Omega^s), \ v \cdot n =0 \text{ on } \Gamma^d},
  &
    V^m&:=H_0(\text{div}, \Omega^d),
  \\
  Q &:= L^2_0(\Omega),
  &
    Q^m&:=L^2_0(\Omega^d).
\end{align*}
 
To shorten notation, we define $\boldsymbol{u}:=(u,u^m)$ as an element
of $\boldsymbol{Z}:=V \times V^m$ and $\boldsymbol{p}:=(p,p^m)$ as an
element of $\boldsymbol{P}:=Q\times Q^m$. We then obtain the following
weak formulation of the dual-porosity-Stokes problem
\cref{eq:Stokes,eq:dualporosity,eq:IC} by a standard process of
testing the equations with $v \in V$, $q\in Q$, $v^m \in V^m$, and
$q^m\in Q^m$, using integration by parts, and applying boundary and
interface conditions: Find
$(\boldsymbol{u}, \boldsymbol{p}) \in \boldsymbol{Z} \times
\boldsymbol{P}$ such that
\begin{subequations}
  \label{eq:mixedform}
  \begin{align}
    \label{eq:mixedform-a}
    a(\boldsymbol{u},\boldsymbol{v})+b(\boldsymbol{v},\boldsymbol{p})&=(f,v)_{\Omega^s} && \forall \boldsymbol{v} \in \boldsymbol{Z},
    \\
    \label{eq:mixedform-b}
    b(\boldsymbol{u},\boldsymbol{q})-c(\boldsymbol{p},\boldsymbol{q})&=-(g,q)_{\Omega^d} &&\forall \boldsymbol{q} \in \boldsymbol{P},
  \end{align}
\end{subequations}
where the bilinear forms are defined as:
\begin{subequations}
  \begin{align}
    \label{eq:bilinearform-6a}
    a(\boldsymbol{u},\boldsymbol{v})
    :=& (2\mu\epsilon(u),\epsilon(v))_{\Omega^s}+\kappa_f^{-1}(u, v)_{\Omega^d}+\kappa_m^{-1}(u^m, v^m )_{\Omega^d}
    \\ \nonumber
      &+ \langle \alpha\mu\kappa_f^{-1/2} (u^s)^t, (v^s)^t\rangle_{\Gamma^I},        
    \\
    \label{eq:bilinearform-6b}
    b(\boldsymbol{u},\boldsymbol{q})
    :=& -(q,\nabla\cdot u)_{\Omega} -(q^m,\nabla\cdot u^m )_{\Omega^d},
    \\
    \label{eq:bilinearform-6c}
    c(\boldsymbol{p},\boldsymbol{q})
    :=&\sigma\kappa_m(p-p^m,q-q^m)_{\Omega^d}.
  \end{align}
\end{subequations}

Classical saddle point theory will be used to prove well-posedness of
\cref{eq:mixedform}, i.e., we show that $a(\cdot,\cdot)$,
$b(\cdot,\cdot)$, and $c(\cdot, \cdot)$ satisfy the conditions of
\cite[Theorem 4.3.1]{Boffi:book}. For this, we first define the
following norms on the velocity space $\boldsymbol{Z}$ and pressure
space $\boldsymbol{P}$:
\begin{equation*}
  \begin{split}
    \norm[0]{\boldsymbol{u}}_Z
    &:=\del[1]{\norm[0]{\epsilon(u^s)}^2_{\Omega^s} + \norm[0]{u^d}^2_{H(\text{div}, \Omega^d)}
      +\norm[0]{u^m}^2_{H(\text{div},\Omega^d)} +\norm[0]{(u^s)^t}_{\Gamma^I}^2}^{1/2},
    \\
    \norm[0]{\boldsymbol{p}}_P
    &:= \del[0]{\norm[0]{p}^2_{\Omega} + \norm[0]{p^m}^2_{\Omega^d}}^{1/2}.  
  \end{split}
\end{equation*}
The following lemma proves an inf-sup condition.

\begin{lemma}
  \label{lem:b-infsup}
  There exists a constant $C > 0$, depending only on $\Omega$, such
  that
  \begin{equation*}
    \forall \boldsymbol{q} \in \boldsymbol{P}, \qquad
    C \norm[0]{\boldsymbol{q}}_P
    \leq
    \sup_{\boldsymbol{v}\in \boldsymbol{Z}\backslash \cbr[0]{(0,0)}}
    \frac{b(\boldsymbol{v},\boldsymbol{q})}{\norm[0]{\boldsymbol{v}}_Z}.
  \end{equation*}
\end{lemma}
\begin{proof}
  Let $\boldsymbol{q} = (q, q^m) \in \boldsymbol{P}$. Since
  $q\in L^2_0(\Omega)$, by the standard inf-sup condition
  \cite[p.463]{Boffi:book}, there exists $v \in H^1_0(\Omega)$ such
  that $\nabla\cdot v=-q$ in $\Omega$ and
  $\norm[0]{v}_{1,\Omega}\leq c \norm[0]{q}_{\Omega}$. Furthermore,
  since $q^m \in L^2_0(\Omega^d)$, by the inf-sup condition for
  $H_0(\text{div}, \Omega^d)$ and $L_0^2(\Omega^d)$
  \cite[p.117-118]{Ern:book}, there exists
  $v^m \in H_0(\text{div}, \Omega^d)$ such that
  $\nabla \cdot v^m = -q^m$ in $\Omega^d$ and
  $\norm[0]{v^m}_{H(\text{div}, \Omega^d)} \leq C
  \norm[0]{q^m}_{\Omega^d}$. Using the trace inequality
  \cite[(1.24)]{Girault:2009}, we find:
  \begin{align*}
    \norm[0]{\boldsymbol{v}}_Z
    = &\del[1]{\norm[0]{\epsilon(v^s)}^2_{\Omega^s} + \norm[0]{v^d}_{H(\text{div}, \Omega^d)}
      + \norm[0]{v^m}^2_{H(\text{div},\Omega^d)} + \norm[0]{(v^s)^t}_{\Gamma^I}^2}^{\frac{1}{2}}
    \\
    \leq &C \del[1]{ \norm[0]{q}^2_{\Omega} + \norm[0]{q^m}^2_{\Omega^d} }^{\frac{1}{2}} = C \norm[0]{\boldsymbol{q}}_P.
  \end{align*}
  Combining the aforementioned results,
  \begin{equation*}
    \frac{b(\boldsymbol{v},\boldsymbol{q})}{\norm[0]{\boldsymbol{v}}_Z}
    = \frac{\norm[0]{\boldsymbol{q}}_P^2}{\norm[0]{\boldsymbol{v}}_Z}
    \geq C\norm[0]{\boldsymbol{q}}_P,
  \end{equation*}
  so that the result follows.
\end{proof}

Let us introduce the operator
$B:\boldsymbol{Z} \mapsto \boldsymbol{P}$ such that
\begin{equation*}
  \del[0]{B\boldsymbol{u}, \boldsymbol{p}}
  = b(\boldsymbol{u}, \boldsymbol{p})
  \quad \forall
  (\boldsymbol{u}, \boldsymbol{p}) \in \boldsymbol{Z} \times \boldsymbol{P}.
\end{equation*}

\begin{lemma}
  \label{lem:a-symcont}
  The bilinear form $a(\cdot,\cdot)$ given in
  \cref{eq:bilinearform-6a} is symmetric and continuous on
  $\boldsymbol{Z} \times \boldsymbol{Z}$, and coercive on
  $\mathrm{Ker}\, B$, i.e.,
  \begin{align*}
    |a(\boldsymbol{u},\boldsymbol{v})|
    &\leq \tilde{C}_b\norm[0]{\boldsymbol{u}}_{Z} \norm[0]{\boldsymbol{v}}_{Z}
    && \forall \boldsymbol{u},\boldsymbol{v} \in \boldsymbol{Z},
    \\   
    a(\boldsymbol{u},\boldsymbol{u})
    &\geq \tilde{C}_c \norm[0]{\boldsymbol{u}}_{Z}^2
    && \forall \boldsymbol{u}\in \mathrm{Ker}\, B,
  \end{align*}
  where
  $\tilde{C}_{b}=\max\cbr[0]{2\mu,\kappa_f^{-1},\alpha\mu\kappa_f^{-1/2},\kappa_m^{-1}}$
  and
  $\tilde{C}_c=\min\cbr[0]{2\mu,\kappa_{f}^{-1},\alpha\mu\kappa_f^{-1/2},\kappa_m^{-1}}$.
\end{lemma}
\begin{proof}
  Continuity follows by using the Cauchy--Schwarz inequality twice on
  the definition of $a(\cdot, \cdot)$ (see \cref{eq:bilinearform-6a})
  \begin{align*}
    |a(\boldsymbol{u},\boldsymbol{v})| 
    \leq
    &2\mu\norm[0]{\epsilon(u)}_{\Omega^s}\norm[0]{\epsilon(v)}_{\Omega^s}
      + \kappa_f^{-1}\norm[0]{u}_{\Omega^d}\norm[0]{v}_{\Omega^d}
      + \kappa_m^{-1}\norm[0]{u^m}_{\Omega^d}\norm[0]{v^m}_{\Omega^d}
    \\
      &+ \alpha\mu\kappa_f^{-1/2} \norm[0]{(u^s)^t}_{\Gamma^I} \norm[0]{(v^s)^t}_{\Gamma^I}
    \\
    \leq
    & (2\mu\norm[0]{\epsilon(u)}^2_{\Omega^s}
      + \kappa_f^{-1}\norm[0]{u}_{\Omega^d}^2
      + \alpha\mu\kappa_f^{-1/2} \norm[0]{(u^s)^t}_{\Gamma^I}^2
      + \kappa_m^{-1}\norm[0]{u^m}_{\Omega^d}^2)
    \\
    & \cdot (2\mu\norm[0]{\epsilon(v)}^2_{\Omega^s}
      + \kappa_f^{-1}\norm[0]{v}_{\Omega^d}^2
      + \alpha\mu\kappa_f^{-1/2}\norm[0]{(v^s)^t}_{\Gamma^I}^2
      + \kappa_m^{-1}\norm[0]{v^m}_{\Omega^d}^2)
    \\
    \leq
    & \tilde{C}_{b} \norm[0]{\boldsymbol{u}}_{Z}\norm[0]{\boldsymbol{v}}_{Z}.
  \end{align*}
  Coercivity of $a(\cdot, \cdot)$ on $\text{Ker}\, B$ follows since
  for $\boldsymbol{u} \in \text{Ker}\, B$,
  \begin{align*}
    a(\boldsymbol{u},\boldsymbol{u})
    &=2\mu\norm[0]{\epsilon(u)}_{\Omega^s}^2 + \kappa_f^{-1}\norm[0]{u}_{\Omega^d}^2 + \kappa_m^{-1}\norm[0]{u^m}_{\Omega^d}^2 
      +\alpha\mu\kappa_f^{-1/2}  \norm[0]{(u^s)^t}_{\Gamma^I}^2
    \\
    & \geq
      \tilde{C}_c\del[0]{
      \norm[0]{\epsilon(u)}_{\Omega^s}^2 + \norm[0]{u}_{\Omega^d}^2
      + \norm[0]{u^m}_{\Omega^d}^2 + \norm[0]{(u^s)^t}_{\Gamma^I}^2}
      = \tilde{C}_c \norm[0]{\boldsymbol{u}}_Z^2.
  \end{align*}
\end{proof}

\begin{lemma}
  \label{lem:b-cont}
  The bilinear form $b(\cdot,\cdot)$, defined in
  \cref{eq:bilinearform-6b}, is continuous on
  $\boldsymbol{Z} \times \boldsymbol{P}$, that is,
  \begin{equation*}
    |b(\boldsymbol{u},\boldsymbol{q})| 
    \leq \norm[0]{\boldsymbol{u}}_Z \norm[0]{\boldsymbol{q}}_P
    \quad \forall (\boldsymbol{u},\boldsymbol{q}) \in \boldsymbol{Z} \times  \boldsymbol{P}.
  \end{equation*}
\end{lemma}
\begin{proof}
  By the Cauchy--Schwarz and Korn's inequalities,
  \begin{align*}
    &|b(\boldsymbol{u},\boldsymbol{q})|
    \leq \norm[0]{q^s}_{\Omega^s} \norm[0]{\nabla \cdot u^s}_{\Omega^s}
      + \norm[0]{q^d}_{\Omega^d}\norm[0]{\nabla \cdot u^d}_{\Omega^d}
      + \norm[0]{q^m}_{\Omega^d}\norm[0]{\nabla \cdot u^m}_{\Omega^d}
    \\
    &\leq \norm[0]{q^s}_{\Omega^s}\norm[0]{\epsilon(u^s)}_{\Omega^s}
      + \norm[0]{q^d}_{\Omega^d}\norm[0]{u^d}_{H(\text{div},\Omega^d)}
      + \norm[0]{q^m}_{\Omega^d}\norm[0]{u^m}_{H(\text{div},\Omega^d)}
    \\
    &\leq (\norm[0]{q^s}_{\Omega^s}^2 + \norm[0]{q^d}^2_{\Omega^d}
      + \norm[0]{q^m}^2_{\Omega^d})^{\frac12} (\norm[0]{\epsilon(u^s)}_{\Omega^s}^2
      + \norm[0]{u^d}_{H(\text{div}, \Omega^d)}^2 + \norm[0]{u^m}_{H(\text{div},\Omega^d)}^2)^{\frac12}
    \\
    &\leq \norm[0]{\boldsymbol{q}}_P \norm[0]{\boldsymbol{u}}_Z.
  \end{align*}
\end{proof}

\begin{lemma}
  \label{lem:c-scsd}
  The bilinear form $c(\cdot, \cdot)$, given in
  \cref{eq:bilinearform-6c}, is symmetric, continuous, and positive
  semidefinite on 
    $\boldsymbol{P} \times \boldsymbol{P}$.
\end{lemma}
\begin{proof}
  It is clear from the definition of $c(\cdot, \cdot)$ that this
  bilinear form is symmetric. Continuity of $c(\cdot, \cdot)$ follows
  by using the Cauchy--Schwarz inequality:
  \begin{align*}
    c(\boldsymbol{p},\boldsymbol{q})
    &= \sigma\kappa_m(p-p^m,q-q^m)_{\Omega^d}
     \leq \sigma\kappa_m\norm[0]{p-p^m}_{\Omega^d}\norm[0]{q-q^m}_{\Omega^d}
    \\
    &\leq \sigma\kappa_m(\norm[0]{p}_{\Omega^d} + \norm[0]{p^m}_{\Omega^d})(\norm[0]{q}_{\Omega^d} + \norm[0]{q^m}_{\Omega^d})
    \\
    &\leq 2\sigma\kappa_m(\norm[0]{p}^2_{\Omega^d} + \norm[0]{p^m}^2_{\Omega^d})^{1/2}(\norm[0]{q}^2_{\Omega^d} + \norm[0]{q^m}^2_{\Omega^d})^{1/2}
    \\
    &\leq 2\sigma^*\kappa_m \norm[0]{\boldsymbol{p}}_P \norm[0]{\boldsymbol{q}}_P,
  \end{align*}
  while
  \begin{equation*}
    c(\boldsymbol{p}, \boldsymbol{p})
    =\sigma \kappa_m \norm[0]{p-p^m}^2_{\Omega^d} \geq 0 \quad
    \forall \boldsymbol{p}\in \boldsymbol{P},
  \end{equation*}
  shows that $c(\cdot, \cdot)$ is positive semidefinite.
\end{proof}

We now state the main result of this section.

\begin{theorem}
  Given $f \in [L^2(\Omega^s)]^{\text{dim}}$ and $g\in L^2(\Omega^d)$,
  the weak formulation \cref{eq:mixedform} has a unique
  solution. Moreover,
  \begin{equation*}
    \norm[0]{\boldsymbol{u}}_Z + \norm[0]{\boldsymbol{p}}_{P}
    \leq C \del[1]{\norm[0]{f}_{\Omega^s} + \norm[0]{g}_{\Omega^d}}.
  \end{equation*}
\end{theorem}
\begin{proof}
  This is an immediate consequence of \cite[Theorem
  4.3.1]{Boffi:book},
  \Cref{lem:b-infsup,lem:a-symcont,lem:b-cont,lem:c-scsd}, and that
  $\text{Ker}\,B^T$ is the zero set $\cbr[0]{0}$ by
  definition of $Q$ and $Q^m$.
\end{proof}

%---------------------------------------------------------------------
\section{The HDG method}
\label{sec:hdg}

%---------------------------------------------------------------------
\subsection{Notation}
\label{ss:not-hdg}

The HDG method presented here for
\cref{eq:Stokes,eq:dualporosity,eq:IC} is an extension of the HDG
method for the Stokes--Darcy problem as presented in
\cite{Cesmelioglu:2020}. Let $\mathcal{T}^j$ be a triangulation of
$\Omega^j$, $j=s,d$, such that $\mathcal{T}^s$ and $\mathcal{T}^d$
match at the interface $\Gamma^I$ and let
$\mathcal{T}:=\mathcal{T}^s\cup \mathcal{T}^d$.  We denote by $h_K$
the diameter of $K$, for $K\in \mathcal{T}$, and define
$h=\max_{K\in \mathcal{T}} h_K$. For $j=s,d$, let $\mathcal{F}^j_i$
denote the set of all interior facets in $\Omega^j$, let
$\mathcal{F}^j_b$ denote the set of all facets on the boundary
$\Gamma^j$, and $\mathcal{F}^I$ denote the set of all facets that lie
on $\Gamma^I$. We define
$\mathcal{F}^j:=\mathcal{F}_i^j\cup \mathcal{F}_b^j \cup
\mathcal{F}^I$, and the set of all facets in $\overline{\Omega}$ is
denoted by $\mathcal{F}$.  Let $\Gamma_0^j$ be the union of all facets
in $\overline{\Omega}^j$, for $j=s,d$, and let $\Gamma_0$ be the union
of all facets in $\overline{\Omega}$ .

We define the following discrete velocity and pressure spaces on
$\Omega$ and $\Omega^j$, $j=s,d$:
\begin{align*}
  V_h
  &:=\cbr[1]{v_h\in [L^2(\Omega)]^{\text{dim}}: v_h \in [P_k(K)]^{\text{dim}} \quad \forall K\in \mathcal{T} }, 
  \\
  V_h^j
  &:=\cbr[1]{v_h\in [L^2(\Omega^j)]^{\text{dim}}: v_h \in [P_k(K)]^{\text{dim}} \quad \forall K\in \mathcal{T}^j } \quad j=s,d,
  \\
  Q_h
  &:=\cbr[1]{q_h\in L^2_0(\Omega): q_h\in P_{k-1}(K)\quad \forall K\in \mathcal{T} },
  \\
  Q_h^j
  &:=\cbr[1]{q_h\in L^2(\Omega^j): q_h\in P_{k-1}(K)\quad \forall K\in \mathcal{T}^j },\quad j=s,d,
  \\
  Q_h^m
  &:=\cbr[1]{r_h\in L^2_0(\Omega^d): r_h\in P_{k-1}(K)\quad \forall K\in \mathcal{T}^d },
\end{align*}
where $P_k(K)$ denotes the polynomial space of total degree at most
$k$ in $K$. We also define the following discrete trace spaces for the
Stokes velocity and pressure on $\Gamma_0^s$, and pressures in the
microfractures and the matrix on $\Gamma_0^d$:
\begin{align*}
  \bar{V}_h
  &:=\cbr[1]{ \bar{v}_h\in [L^2(\Gamma_0^s)]^{\text{dim}}: \bar{v}_h\in [P_k(F)]^{\text{dim}}
    \quad \forall F \in \mathcal{F}^s, \quad \bar{v}_h=0 \text{ on } \Gamma^s},
  \\
  \bar{Q}_h^j
  &:=\cbr[1]{\bar{q}_h^j\in L^2(\Gamma_0^j): \bar{q}_h^j\in P_k(F)\quad \forall F \in \mathcal{F}^j }, \quad j=s,d.
\end{align*}
For notational convenience, we further define
\begin{align*}
  \boldsymbol{V}_h &:= V_h\times \bar{V}_h,
  & \boldsymbol{Q}_h &:= Q_h\times\bar{Q}_h^s\times\bar{Q}_h^d,
  & \boldsymbol{Q}_h^j &:= Q_h^j\times\bar{Q}_h^j,\, j=s,d,  
  \\
  \boldsymbol{Q}_h^m &:= Q_h^m\times \bar{Q}_h^d,
  & \boldsymbol{Z}_h &:= \boldsymbol{V}_h\times V_h^d,
  & \boldsymbol{P}_h& :=\boldsymbol{Q}_h\times\boldsymbol{Q}_h^m,
  &&
\end{align*}
and set
$\boldsymbol{v}_h:=(v_h,\bar{v}_h, v_h^m)\in \boldsymbol{Z}_h$,
$\boldsymbol{q}_h:=(q_h,\bar{q}_h^s,\bar{q}_h^d,q_h^m,\bar{q}_h^m)\in
\boldsymbol{P}_h$, and $\boldsymbol{q}_h^j:=(q_h^j,\bar{q}_h^j)$,
$j=s,d, m$.

The spaces $\boldsymbol{Z}_h$ and $\boldsymbol{P}_h$ are equipped with
the following norms: 
\begin{align*}
  \norm[0]{\boldsymbol{v}_h}^2_{Z_h}
  &:= \norm[0]{\boldsymbol{v}}^2_{V_h^s} + \norm[0]{v_h^d}_{H(\text{div},\Omega^d)}^2
    + \norm[0]{v_h^m}_{H(\text{div},\Omega^d)}^2 + \norm[0]{\bar{v}_h^t}^2_{\Gamma^I},
  \\
  \norm[0]{\boldsymbol{v}_h}^2_{{Z}_h^*}
  &:= \norm[0]{\boldsymbol{v}_h}_{{Z}_h} + \sum_{K\in \mathcal{T}^s} h_K^2 |v_h|^2_{2,K},
  \\
  \norm[0]{\boldsymbol{q}_h}_{P_h}^2
  &:= \sum_{j=s,d,m}\norm[0]{\boldsymbol{q}_h^j}_{Q_h^j}^2,
\end{align*}
where in the Stokes and dual-porosity subdomains we further define:
\begin{align*}
  \|\boldsymbol{v}_h\|^2_{V_h^s}
  &:= \|\epsilon(v_h^s)\|^2_{\Omega^s}+\sum_{K\in \mathcal{T}^s}h_K^{-1}\|v_h-\bar{v}_h\|^2_{\partial K},
  \\
  \|\boldsymbol{v}_h\|^2_{V_h^{s,*}}
  &:= \|\boldsymbol{v}_h\|^2_{V_h^s}+\sum_{K\in \mathcal{T}^s} h_K^2 |v_h|^2_{2,K},
  \\
  \|\boldsymbol{q}_h^j\|_{Q_h^j}^2
  &:= \|q_h^j\|^2_{\Omega^j}+\sum_{K\in\mathcal{T}^j}h_K\|\bar{q}_h^j\|^2_{K},\quad j=s,d,
  \\
  \|\boldsymbol{q}_h^m\|_{Q_h^m}^2
  &:= \|q_h^m\|^2_{\Omega^d}+\sum_{K\in\mathcal{T}^d}h_K\|\bar{q}_h^m\|^2_{K}.
\end{align*}

%---------------------------------------------------------------------
\subsection{The discretization}
\label{ss:disc-hdg}

In this section, we present our HDG method for the
dual-porosity-Stokes problem \cref{eq:Stokes,eq:dualporosity,eq:IC}
which couples the IP-HDG discretization \cite{Rhebergen:2017} for the
Stokes equations to a hybridized BDM discretization
\cite{Arnold:1985,Boffi:book,Du:book} for the dual-porosity problem.

Let us first define
$(\cdot,\cdot )_{\Omega^j}:=\sum_{K\in \mathcal{T}^j}(\cdot, \cdot)_K$
and
$\langle \cdot,\cdot \rangle_{\partial
  \mathcal{T}^j}:=\sum_{K\in\mathcal{T}^s}\langle
\cdot,\cdot\rangle_{\partial K}$ for $j=s,d$. The HDG method is given
by: \\
Find
$(\boldsymbol{u}_h,\boldsymbol{p}_h):=((u_h,\bar{u}_h,u^m_h),
(p_h,\bar{p}_h^s,\bar{p}_h^d, p_h^m,\bar{p}_h^m))\in
\boldsymbol{Z}_h\times \boldsymbol{P}_h$ such that:
\begin{subequations}
  \label{eq:HDG-dualporosityproblem-18}
  \begin{align}
    \label{eq:HDG-dualporosityproblem-18-1}
    a_h(\boldsymbol{u}_h,\boldsymbol{v}_h)+ b_h(\boldsymbol{p}_h,\boldsymbol{v}_h)
    &=( f, v_h)_{\Omega^s}
    && \forall \boldsymbol{v}_h:=(v_h,\bar{v}_h,v_h^m)\in \boldsymbol{Z}_h,
    \\
    \label{eq:HDG-dualporosityproblem-18-2}
    b_h(\boldsymbol{q}_h,\boldsymbol{u}_h)-c_h(\boldsymbol{p}_h,\boldsymbol{q}_h)
    &=-( g, q_h)_{\Omega^d}
    &&\forall \boldsymbol{q}_h:=(q_h,\bar{q}_h^s,\bar{q}_h^d, q_h^m,\bar{q}_h^m) \in \boldsymbol{P}_h,
  \end{align}
\end{subequations}
where 
\begin{align*}
  a_h(\boldsymbol{u},\boldsymbol{v})
  &:=a_h^s(\boldsymbol{u},\boldsymbol{v})+a_h^d(u,v)+a_h^m(u^m,v^m)+a_h^I(\bar{u},\bar{v}),
  \\
  b_h(\boldsymbol{p},\boldsymbol{v})
  &:=\sum_{j=s,d}[b_h^j(\boldsymbol{p},v)+b_h^{I,j}(\bar{p}^j,\bar{v})]+b_h^d(\boldsymbol{p}^m,v^m),
  \\
  c_h(\boldsymbol{p},\boldsymbol{q})
  &:=\sigma\kappa_m(p^m-p,q^m-q)_{\Omega^d}, 
\end{align*}
and
\begin{align*}
  a_h^s(\boldsymbol{u},\boldsymbol{v})
  :=&( 2\mu\epsilon(u),\epsilon(v))_{\Omega^s}-\langle 2\mu\epsilon(u)n^s,v-\bar{v}\rangle_{\partial \mathcal{T}^s}
      -\langle 2\mu\epsilon(v)n^s,u-\bar{u}\rangle_{\partial \mathcal{T}^s}
  \\
    &+\sum_{K\in \mathcal{T}^s}2\beta\mu h_K^{-1}\langle u-\bar{u},v-\bar{v}\rangle_{\partial K},
  \\
  a_h^d(u,v)
  :=&( \kappa_f^{-1}u,v)_{\Omega^d},
  \\
  a_h^m(u^m,v^m)
  :=&( \kappa_m^{-1}u^m, v^m )_{\Omega^d},
  \\
  a_h^I(\bar{u},\bar{v})
  :=&\langle \alpha\mu \kappa_f^{-1/2}\bar{u}^t,\bar{v}^t\rangle_{\Gamma^I},
  \\
  b_h^j(\boldsymbol{p},v)
  :=&-( p,\nabla \cdot v)_{\Omega^j}+\langle\bar{p}^j, v\cdot n\rangle_{\partial \mathcal{T}^j},\quad j=s,d,
  \\
  b_h^{I,j}(\bar{p}^j,\bar{v})
  :=&-\langle\bar{p}^j,\bar{v}\cdot n^j\rangle_{\Gamma^I}, \quad j=s,d.
\end{align*}

The following lemma shows that $u_h$ is
$H(\text{div};\Omega)$-conforming, that $u_h^m$ is
$H(\text{div};\Omega^d)$-conforming, that $u_h^s$ is pointwise
divergence-free on the elements in $\mathcal{T}^s$,  that \cref{eq:dualporosity-b} is satisfied
pointwise on the elements in $\mathcal{T}^d$ up to the error of the
$L^2$-projection of the source term $g$ into $Q_h^d$, and that
\cref{eq:dualporosity-d} is satisfied pointwise on the elements in
$\mathcal{T}^d$.
\begin{lemma}
  \label{lem:Mass conservation}
  The solution to \cref{eq:HDG-dualporosityproblem-18} satisfies:
  \begin{subequations}
    \label{eq:Numerical properties}
    \begin{align}
      \label{eq:Numerical properties a}
      \jump{u_h^j\cdot n}&=0 && \forall x \in F, \quad \forall F \in \mathcal{F}^j\backslash \mathcal{F}^I, \quad j=s,d,
      \\
      \label{eq:Numerical properties b}
      u_h^j \cdot n&= \bar{u}_h \cdot n &&\forall x \in F,  \quad \forall F \in \mathcal{F}^I,\quad j=s,d,
      \\
      \label{eq:Numerical properties c}
      \jump{u_h^m\cdot n}&=0 &&\forall x \in F,  \quad \forall F \in \mathcal{F}^d,
      \\
       \label{eq:Numerical properties f}
      \nabla\cdot u_h^s&=0 && \forall x \in K, \quad \forall K \in \mathcal{T}^s,
      \\
      \label{eq:Numerical properties d}
      \sigma \kappa_m(p^d_h-p^m_h)+\nabla \cdot u^d_h&=\Pi_Q^d g && \forall x \in K, \quad \forall K \in \mathcal{T}^d,
      \\
      \label{eq:Numerical properties e}
      \sigma \kappa_m(p^m_h-p^d_h)+\nabla \cdot u^m_h&=0 && \forall x \in K, \quad \forall K \in \mathcal{T}^d,
    \end{align}
  \end{subequations}
  where $\jump{\cdot}$ is the standard jump operator and $\Pi_Q^d$
  denotes the $L^2$-projection onto $Q_h^d$.
\end{lemma}
\begin{proof}
  Choosing $\boldsymbol{v}_h=0$, $q_h=0$, and $q_h^m=0$ in
  \cref{eq:HDG-dualporosityproblem-18} we obtain:
  \begin{equation*}
    0 = \sum_{j=s,d}\langle \Bar{q}_h^j,u_h\cdot n\rangle_{\partial \mathcal{T}^j}
    + \langle \Bar{q}_h^m, u_h^m\cdot n\rangle_{\partial \mathcal{T}^d}
    - \sum_{j=s,d} \langle \bar{q}_h^j, \Bar{u}_h\cdot n^j \rangle_{\Gamma^I},
  \end{equation*}
  for all
  $(\Bar{q}_h^s,\Bar{q}_h^d,\Bar{q}_h^m)\in\Bar{Q}_h^s \times
  \Bar{Q}_h^d \times \Bar{Q}_h^m$. Therefore,
  \begin{equation} 
    \label{eq:properties}
    0 = \sum_{j=s,d}\Big[\sum_{F\in\mathcal{F}^j_i\cup \mathcal{F}^j_b}\langle\Bar{q}_h^j,\jump{u_h\cdot n}\rangle_F
    + \sum_{F\in \mathcal{F}^I}\langle\bar{q}_h^j,(u_h^j-\Bar{u}_h)\cdot n^j\rangle_F\Big]
    + \sum_{F\in\mathcal{F}^d}\langle\Bar{q}_h^m,\jump{u_h^m\cdot n}\rangle_F.
  \end{equation}
  \Cref{eq:Numerical properties a,eq:Numerical properties
    b,eq:Numerical properties c} follow by setting, for $j=s,d$,
  \begin{equation*}
    \Bar{q}_h^j=
    \begin{cases}
      \jump{u_h\cdot n}  & \text{ on } F\in\mathcal{F}^j_i\cup \mathcal{F}^j_b,
      \\
      (u_h^j-\Bar{u}_h)\cdot n^j & \text{ on } F\in \mathcal{F}^I,
    \end{cases}
  \end{equation*}
  and $\bar{q}_h^m=\jump{u_h^m\cdot n} \text{ on } F\in\mathcal{F}^d$
  in \cref{eq:properties}. To prove \cref{eq:Numerical properties
    d,eq:Numerical properties e,eq:Numerical properties f}, set
  $\boldsymbol{v}_h=\boldsymbol{0}$, $q_h^m=0$, and $\Bar{q}_h^j=0$,
  $j=s,d,m$ in \cref{eq:HDG-dualporosityproblem-18}. 
  
    Then, selecting
    \begin{equation*}
      q_h=
      \begin{cases}
        \nabla \cdot u_h^s & \text{ in   } \Omega^s,
        \\
        \sigma \kappa_m(p_h^d-p_h^m)+ \nabla \cdot u_h^d-\Pi_Q^dg &\text{ in   }\Omega^d,
      \end{cases}  
    \end{equation*}
    gives \cref{eq:Numerical properties f,eq:Numerical properties d}.
  Finally, setting
  $\boldsymbol{v}_h=\boldsymbol{0}$, $q_h=0$, $\Bar{q}_h^j=0$,
  $j=s,d,m$, and
  $q_h^m=\sigma \kappa_m(p_h^m-p_h^d)+ \nabla \cdot u_h^m$ in
  \cref{eq:HDG-dualporosityproblem-18}, we obtain \cref{eq:Numerical
    properties e}.
\end{proof}

%---------------------------------------------------------------------
\subsection{Consistency and well-posedness of the HDG method}
\label{ss:con-wp-hdg}

The next lemma shows that the HDG method
\cref{eq:HDG-dualporosityproblem-18} is a consistent discretization of
the dual-porosity-Stokes problem
\cref{eq:Stokes,eq:dualporosity,eq:IC}.

\begin{lemma}[Consistency]
    \label{lem:consistency}
  If $(u, u^m, p, p^m)$ solves the dual-porosity-Stokes problem
  \cref{eq:Stokes,eq:dualporosity,eq:IC}, $\bar{u}$ denotes the
  trace of $u$ on the mesh skeleton, and $\bar{p}^j$ denotes the
  trace of $p^j$ ($j=s,d,m$) on the mesh skeleton, then
  $(u,\bar{u},u^m)$ and $(p,\bar{p}^s,\bar{p}^d,p^m,\bar{p}^m)$
  satisfy \cref{eq:HDG-dualporosityproblem-18}.  
\end{lemma}
\begin{proof}
  The proof follows the same argument as in \cite[Lemma
  1]{Cesmelioglu:2020}. Using smoothness of $u$, single-valuedness of
  $\bar{v}_h$, that $\bar{v}_h=0$ on $\Gamma^s$, and \cref{eq:IC-d} we
  get:
  \begin{align*}
    a_h&((u,\bar{u},u^m),\boldsymbol{v}_h)
    =( 2\mu\epsilon(u),\epsilon(v_h))_{\Omega^s}
       -\langle2\mu\epsilon(u)n^s,v_h-\bar{v}_h\rangle_{\partial \mathcal{T}^s}
      \\
     &  +\langle\alpha \mu\kappa_f^{-1/2}(u^s)^t,\bar{v}_h^t\rangle_{\Gamma^I}
   +( \kappa_f^{-1}u, v_h)_{\Omega^d}+( \kappa_m^{-1}u^m, v_h^m)_{\Omega^d}
    \\
    =&- ( \nabla \cdot (2\mu \epsilon(u)),v_h)_{\Omega^s}
       -\langle\alpha \mu \kappa_f^{-1/2}(u^s)^t, \bar{v}_h^t \rangle_{\Gamma^I}
       +\langle2\mu(n^s\cdot \epsilon(u)n^s)n^s, \bar{v}_h \rangle_{\Gamma^I}
    \\
     &+\langle \alpha \mu \kappa_f^{-1/2} (u^s)^t, \bar{v}_h^t\rangle_{\Gamma^I}
       + (\kappa_f^{-1}u, v_h)_{\Omega^d}+( \kappa_m^{-1} u^m, v_h^m )_{\Omega^d}
    \\
    =&- ( \nabla \cdot (2\mu\epsilon(u)), v_h)_{\Omega^s}
       +\langle 2\mu(n^s\cdot \epsilon(u)n^s)n^s, \bar{v}_h \rangle_{\Gamma^I}
      \\
     & + (\kappa_f^{-1}u, v_h)_{\Omega^d}+( \kappa_m^{-1} u^m, v_h^m )_{\Omega^d}.
  \end{align*}
  Furthermore, after integration-by-parts,
  \begin{align*}
    &\sum_{j=s,d}
      [b_h^j((p,\bar{p}^s,\bar{p}^d,p^m,\bar{p}^m),v_h)
      + b_h^{I,j}(\bar{p}^j,\bar{v}_h)]
      + b_h^d((p^m,\bar{p}^m),v_h^m)
    \\
    =&-( p,\nabla \cdot v_h)_{\Omega^s}
       +\langle p, v_h\cdot n\rangle_{\partial \mathcal{T}^s}
       -\langle p^s,\bar{v}_h\cdot n^s\rangle_{\Gamma^I}
       -( p,\nabla \cdot v_h)_{\Omega^d}+
       \langle p, v_h\cdot n\rangle_{\partial \mathcal{T}^d}
    \\
    &-\langle p^d,\bar{v}_h\cdot n^d\rangle_{\Gamma^I}
      -( p^m,\nabla\cdot v_h^m)_{\Omega^d}
      +\langle p^m, v_h^m\cdot n^d \rangle_{\partial \mathcal{T}^d}
    \\
    =& ( \nabla p,v_h)_{\Omega}
       +\langle(p^d-p^s),\bar{v}_h\cdot n \rangle_{\Gamma^I}
       +( \nabla p^m, v_h^m)_{\Omega^d}.
  \end{align*}
  Combining the above two results:
  \begin{align*}
    a_h&((u,\bar{u},u^m),\boldsymbol{v}_h)
    +\sum_{j=s,d}[b_h^j((p,\bar{p}^s,\bar{p}^d,p^m,\bar{p}^m),v_h)
      +b_h^{I,j}(p^j,\bar{v}_h)]+b_h^d(\boldsymbol{p}^m,v_h^m)
    \\
    &+\sum_{j=s,d}\big(b_h^j(\boldsymbol{q}_h,u)+b_h^{I,j}(\bar{q}_h^j,u)\big)
      +b_h^d(\boldsymbol{q}_h^m,u^m)-c_h((p,\bar{p}^s,\bar{p}^d,p^m,\bar{p}^m),\boldsymbol{q}_h)
    \\
    =&- ( \nabla \cdot (2\mu \epsilon(u)),v_h)_{\Omega^s}
       + \langle 2\mu(n^s\cdot \epsilon(u)n^s)n^s,\bar{v}_h \rangle_{\Gamma^I}
       + (\kappa_f^{-1}u, v_h)_{\Omega^d}
       \\
       &+ (\kappa_m^{-1} u^m, v_h^m)_{\Omega^d}
    + ( \nabla p, v_h)_{\Omega}+\langle p^d-p^s,\bar{v}_h\cdot n\rangle_{\Gamma^I}
      + ( \nabla p^m, v_h^m)_{\Omega^d}
    \\
    &- ( q_h,\nabla \cdot u)_{\Omega}
    + \langle \bar{q}_h^s, u \cdot n^s \rangle_{\partial \mathcal{T}^s}
      + \langle \bar{q}_h^d, u \cdot n^d\rangle_{\partial \mathcal{T}^d}
      \\
      &- \langle \bar{q}_h^s,u\cdot n^s\rangle_{\Gamma^I}
      - \langle \bar{q}_h^d,u\cdot n^d\rangle_{\Gamma^I}
    \\
    &- ( q_h^m,\nabla\cdot u^m)_{\Omega^d}
      + \langle \bar{q}_h^m, u^m\cdot n^d \rangle_{\partial \mathcal{T}^d}
      - ( \sigma \kappa_m (p^m-p), q_h^m-q_h)_{\Omega^d}
    \\
    =&- ( \nabla \cdot (2\mu \epsilon(u)) - \nabla p, v_h )_{\Omega^s}
       +(\kappa_f^{-1}u+\nabla p, v_h)_{\Omega^d}  
    \\
    &+\langle \big(2\mu(n^s\cdot \epsilon(u)n^s) + p^d-p^s)n^s,\bar{v}_h \rangle_{\Gamma^I}
      -( q_h,\nabla \cdot u+ \sigma \kappa_m (p-p^m) )_{\Omega^d}
    \\
    & -( q^m_h, \nabla \cdot u^m+ \sigma \kappa_m (p^m-p) )_{\Omega^d}
    \\
    =&(f, v_h)_{\Omega^s} -(g, q_h)_{\Omega^d},
  \end{align*}
  where we used smoothness of $u$ and $u^m$, single-valuedness of
  $\bar{q}_h^j$, $j=s,d,m$, and
  \cref{eq:Stokes,eq:dualporosity,eq:IC}.
\end{proof}

The next two lemmas show coercivity of $a_h(\cdot, \cdot)$ and
boundedness of $a_h^s(\cdot, \cdot)$ and $a_h(\cdot, \cdot)$.

\begin{lemma}[Coercivity]
  \label{lem:Coercivity a}
  Let
  $C_e=\min\{\mu(1-C_{\text{tr}}^2/\beta)\min(1,\beta),\alpha
  \kappa_f^{-1/2},\kappa_f^{-1},\kappa_m^{-1}\}$ where $C_{\text{tr}}>0$ is a
  constant of discrete trace inequality independent of $h, \mu$, $\kappa_f$, $\kappa_m$, and
  $\sigma$. Then for sufficiently large penalty parameter $\beta>0$,
  \begin{equation*}
    a_h(\boldsymbol{v}_h,\boldsymbol{v}_h)
    \geq
    C_e(\|\boldsymbol{v}_h\|_{V_h^s}^2
    +\|\bar{v}_h^t\|_{\Gamma^I}+\|v_h\|^2_{\Omega^d}+\|v_h^m\|^2_{\Omega^d}).
  \end{equation*}
  Furthermore, if $\nabla \cdot v_h=\nabla \cdot v_h^m = 0$ in
  $\Omega^d$, then
  \begin{equation*}
    a_h(\boldsymbol{v}_h,\boldsymbol{v}_h)\geq C _e\|\boldsymbol{v}_h\|^2_{Z_h}.
  \end{equation*}
\end{lemma}
\begin{proof}
  The result follows the same steps as the proof of \cite[Lemma
  4.2]{Rhebergen:2017}. First note that
  \begin{equation}
    \label{eq:ahvhvhcoerc}
    \begin{split}
      a_h(\boldsymbol{v}_h,\boldsymbol{v}_h)
      =&2\mu\big(\|\epsilon(v_h)\|_{\Omega^s}^2-\langle 2\epsilon(v_h)n^s,v_h-\bar{v}_h\rangle_{\partial \mathcal{T}^s}
      +\beta \sum_{K\in \mathcal{T}^s}h_K^{-1}\|v_h-\bar{v}_h\|_{\partial K}^2\big) 
      \\
      &+\alpha \kappa_f^{-1/2}\|\bar{v}_h^t\|_{\Gamma^I}^2+\kappa_f^{-1}\|v_h\|_{\Omega^d}^2+\kappa_m^{-1}\|v_h^m\|_{\Omega^d}^2.      
    \end{split}
  \end{equation}
  Applying the Cauchy--Schwarz inequality and a discrete trace
  inequality \cite[Lemma 1.46]{DiPietro:book}  on the second term on the right
  hand side, we get:
  \begin{align*}
    \langle 2\epsilon(v_h)n^s,v_h-\bar{v}_h\rangle_{\partial \mathcal{T}^s}
    &\leq 2 \big(\sum_{K\in \mathcal{T}^s} h_K\|\epsilon(v_h)\|_{\partial K}^2\big)^{\frac{1}{2}}
      \big(\sum_{K\in \mathcal{T}^s} h_K^{-1} \|v_h-\bar{v}_h\|_{\partial K}^2\big)^{\frac{1}{2}}
    \\
    &\leq 2C_{\text{tr}} \|\epsilon(v_h)\|_{\Omega^s} \big(\sum_{K\in \mathcal{T}^s}h_K^{-1}\|v_h-\bar{v}_h\|_{\partial K}^2\big)^{\frac{1}{2}}
  \end{align*}
  where $C_{\text{tr}}$ is the discrete trace inequality constant depending on the shape regularity of meshes, $k$, and $\text{dim}$. We refer to \cite{Hestheven:2003} for explicit dependence of $C_{\text{tr}}$ on $k$ and $\text{dim}$. 
  Combine this with \cref{eq:ahvhvhcoerc} and recall the inequality
  $x^2-2\psi x y+y^2\geq (1-\psi^2)(x^2+y^2)/2$, which holds for all
  $x,y \in \mathbb{R}$ and $0<\psi<1$. Choose
  $x=\|\epsilon(v_h)\|_{\Omega^s}$,
  $y^2=\beta\sum_{K\in\mathcal{T}^s}h_K^{-1}\|v-\bar{v}\|^2_{\partial
    K}$, and $\psi=C_{\text{tr}}\beta^{-\frac{1}{2}}$ to find:
  \begin{align*}
    a_h(&\boldsymbol{v}_h, \boldsymbol{v}_h)
    \\
    \geq& 2\mu\Big[\|\epsilon(v_h)\|^2_{\Omega^s}
          -2C_{\text{tr}}\|\epsilon(v_h)\|_{\Omega^s}
          \big(\sum_{K\in \mathcal{T}^s}h_K^{-1}\|v_h-\bar{v}_h\|^2_{\partial K}\big)^{\frac{1}{2}}
    \\
        &
          \hspace{2em}+\beta\sum_{K\in\mathcal{T}^s}h_K^{-1}\|v-\bar{v}\|^2_{\partial K}\Big]
          +\alpha \kappa_f^{-1/2}\|\bar{v}_h^t\|_{\Gamma^I}+\kappa_f^{-1}\|v_h\|^2_{\Omega^d}+\kappa_m^{-1}\|v_h^m\|^2_{\Omega^d}
    \\
    \geq& (\mu(1-C_{\text{tr}}^2/\beta)\big(\|\epsilon(v_h)\|^2_{\Omega^s}+\beta\sum_{K\in \mathcal{T}^s}h_K^{-1}\|v_h-\bar{v}_h\|^2_{\partial K}\big)
          +\alpha \kappa_f^{-1/2}\|\bar{v}_h^t\|_{\Gamma^I}
    \\
        &+\kappa_f^{-1}\|v_h\|^2_{\Omega^d}+\kappa_m^{-1}\|v_h^m\|^2_{\Omega^d}
    \\
    \geq &C_e(\|\boldsymbol{v}_h\|_{V_h^s}^2
           +\|\bar{v}_h^t\|_{\Gamma^I}+\|v_h\|^2_{\Omega^d}+\|v_h^m\|^2_{\Omega^d}),
  \end{align*}
  proving the result.
\end{proof}
\begin{lemma}[Boundedness]
 \label{lem:Continuity of a}
 The bilinear forms $a_h^s(\cdot,\cdot)$ and $a_h(\cdot,\cdot)$
 satisfy 
 \begin{subequations}
   \label{eq:ahboundednessterms}
   \begin{align}
     \label{eq:ahsboundedness}
     a_h^s(\boldsymbol{u}_h,\boldsymbol{v}_h)
     &\leq C_c^s\|\boldsymbol{u}_h\|_{V_h^s}\|\boldsymbol{v}_h\|_{V_h^s}
     && \forall \boldsymbol{u}_h,\boldsymbol{v}_h \in \boldsymbol{Z}_h,
     \\
     \label{eq:ahboundedness}
     a_h(\boldsymbol{u}_h,\boldsymbol{v}_h)
     &\leq C_c\|\boldsymbol{u}_h\|_{Z_h}\|\boldsymbol{v}_h\|_{Z_h}
     && \forall \boldsymbol{u}_h,\boldsymbol{v}_h \in \boldsymbol{Z}_h,
   \end{align}   
 \end{subequations}
 where $C_c^s:=2\mu\max(1+C_{\rm tr}, \beta+C_{\rm tr})$ and
 $C_c:=\max(C_c^s,\kappa_f^{-1},\kappa_m^{-1},\alpha \mu
 \kappa_f^{-1/2})$ and with $C>0$ a constant independent of $\mu$,
 $\kappa_f$, $\kappa_m$, and $\sigma$.
\end{lemma}
\begin{proof}
  The proof is similar to the proof of \cite[Lemma
  3]{Cesmelioglu:2020}. We start by proving
  \cref{eq:ahsboundedness}. Let
  $\boldsymbol{u}_h, \boldsymbol{v}_h \in \boldsymbol{Z}_h$. By the
  Cauchy--Schwarz inequality,
  \begin{equation}
    \label{eq:ahsboundednessstep1}
    \begin{split}
      a_h^s&(\boldsymbol{u}_h,\boldsymbol{v}_h)
      \leq 2\mu\big( \|\epsilon(u_h)\|_{\Omega^s}\|\epsilon({v}_h)\|_{\Omega^s}
      +\sum_{K\in\mathcal{T}^s} \|\epsilon(u_h)\|_{\partial K}\|v_h-\Bar{v}_h\|_{\partial K}
      \\
      &+\sum_{K\in\mathcal{T}^s} \|\epsilon(v_h)\|_{\partial K}\|u_h-\bar{u}_h\|_{\partial K}
      +\sum_{K\in \mathcal{T}^s}\beta h_K^{-1}\|u_h-\Bar{u}_h\|_{\partial K}\|v_h-\Bar{v}_h\|_{\partial K}\big).      
    \end{split}
  \end{equation}
  Note that by the discrete trace inequality \cite[Lemma
  1.46]{DiPietro:book} and the Cauchy--Schwarz inequality,
  \begin{equation}
    \label{eq:ahsboundednessstep2}
    \begin{split}
      &\sum_{K\in\mathcal{T}^s}
       \|\epsilon(u_h)\|_{\partial K}\|v_h-\Bar{v}_h\|_{\partial K}
      +\sum_{K\in\mathcal{T}^s} \|\epsilon(v_h)\|_{\partial K}\|u_h-\bar{u}_h\|_{\partial K}
      \\
      &+\sum_{K\in \mathcal{T}^s}\beta h_K^{-1}\|u_h-\Bar{u}_h\|_{\partial K}\|v_h-\Bar{v}_h\|_{\partial K}
      \\
      \leq
      & \sum_{K\in\mathcal{T}^s} C_{\text{tr}} \|\epsilon(u_h)\|_{K} h_K^{-1/2}\|v_h-\Bar{v}_h\|_{\partial K} 
      +\sum_{K\in\mathcal{T}^s} C_{\text{tr}} \|\epsilon(v_h)\|_{K}h_K^{-1/2}\|u_h-\Bar{u}_h\|_{\partial K}
      \\
      & +\sum_{K\in \mathcal{T}^s}\beta h_K^{-1/2}\|u_h-\Bar{u}_h\|_{\partial K}h_K^{-1/2}\|v_h-\Bar{v}_h\|_{\partial K}
      \\
      \leq
      & C_{\text{tr}}\|\epsilon(u_h)\|_{\Omega^s} (\sum_{K\in\mathcal{T}^s}h_K^{-1}\|v_h-\Bar{v}_h\|^2_{\partial K})^{\frac{1}{2}}
      + C_{\text{tr}} \|\epsilon(v_h)\|_{\Omega^s}(\sum_{K\in\mathcal{T}^s}h_K^{-1}\|u_h-\Bar{u}_h\|^2_{\partial K})^{\frac{1}{2}} 
      \\
      &+\beta (\sum_{K\in \mathcal{T}^s}h_K^{-1}\|u_h-\Bar{u}_h\|_{\partial K}^2)^{1/2}
      (\sum_{K\in \mathcal{T}^s}h_K^{-1}\|v_h-\Bar{v}_h\|^2_{\partial K})^{1/2}\big).      
    \end{split}
  \end{equation}
  Combining \cref{eq:ahsboundednessstep1,eq:ahsboundednessstep2} and
  using the Cauchy--Schwarz inequality once more, we get 
  \begin{align*}
    &a_h^s(\boldsymbol{u}_h,\boldsymbol{v}_h)
    \\
    &\leq 2\mu \big((1+C_{\text{tr}})\|\epsilon(u_h)\|^2_{\Omega^s}+ (\beta+C_{\text{tr}})\sum_{K\in\mathcal{T}^s}h_K^{-1}\|u_h-\Bar{u}_h\|^2_{\partial K}\big)^{\tfrac12}
    \\
    &\quad \times\big((1+C_{\text{tr}})\|\epsilon(v_h)\|^2_{\Omega^s}+ (\beta+C_{\text{tr}})\sum_{K\in\mathcal{T}^s}h_K^{-1}\|v_h-\Bar{v}_h\|^2_{\partial K}\big)^{\tfrac12}
    \leq C_c^s \|\boldsymbol{u}_h\|_{V_h^s}\|\boldsymbol{v}_h\|_{V_h^s}.
  \end{align*}
  To prove \cref{eq:ahboundedness} we use \cref{eq:ahsboundedness} and
  the Cauchy--Schwarz inequality to find:
  \begin{align*}
    a_h(\boldsymbol{u}_h,&\boldsymbol{v}_h)
    = a_h^s(\boldsymbol{u}_h,\boldsymbol{v}_h) + a_h^d(u_h,v_h)+ a_h^m(u_h^m,v_h^m)+ a_h^I(\bar{u}_h,\bar{v}_h) 
    \\
    \leq
     & C_c^s \|\boldsymbol{u}_h\|_{V_h^s}\|\boldsymbol{v}_h\|_{V_h^s}+\kappa_f^{-1}\|u_h\|_{\Omega^d}\|v_h\|_{\Omega^d}
       +\kappa_m^{-1}\|u_h^m\|_{\Omega^d}\|v_h^m\|_{\Omega^d}
       \\
     &  +\alpha \mu \kappa_f^{-1/2}\|\Bar{u}_h^t\|_{\Gamma^I}\|\Bar{v}_h^t\|_{\Gamma^I}.
    \\
    \leq
     &\max(C_c^s,\kappa_f^{-1},\kappa_m^{-1},\alpha \mu \kappa_f^{-1/2})
       \big(\|\boldsymbol{u}_h\|_{V_h^s}^2+\|u_h\|^2_{\Omega^d}
       +\|u_h^m\|^2_{\Omega^d} + \|\Bar{u}_h^t\|^2_{\Gamma^I}\big)^{1/2}
    \\
     & \qquad\times \big(\|\boldsymbol{v}_h\|_{V_h^s}^2+\|v_h\|^2_{\Omega^d}
       +\|v_h^m\|^2_{\Omega^d}+\|\Bar{v}_h^t\|^2_{\Gamma^I}\big)^{1/2}.
    \\
     \leq & C_c\|\boldsymbol{u}_h\|_{Z_h}\|\boldsymbol{v}_h\|_{Z_h}.
  \end{align*}
\end{proof}
\begin{remark} 
  \label{rem:extendedboundedness}
  If the first component of $a_h$ belongs to
  $\boldsymbol{Z}_h+(\widetilde{V}\times \bar{\widetilde{V}}\times \widetilde{V}^m)$, where 
  \begin{align*}
    \widetilde{V}
   := &\{ v\in H(\text{div},\Omega): \, v^s\in [H^2(\Omega)]^{\text{dim}}, \,
      v^d\in [H^1(\Omega^d)]^{\text{dim}}, 
      \\
      &\hspace{3.2cm} v=0 \text{ on } \Gamma^s, \, v\cdot n=0 \text{ on } \Gamma^d\},
    \\
    \widetilde{V}^m
    :=&\{ v^m \in [H^1(\Omega^d)]^{\text{dim}}: \, v^m \cdot n = 0 \text{ on } \partial \Omega^d\},
  \end{align*}
  and $\bar{\widetilde{V}}$ is the trace space of $\widetilde{V}$  on $\Gamma_0^s$, the inequalities in
  \Cref{lem:Continuity of a} become
  \begin{equation*}
    a_h^s(\boldsymbol{u},\boldsymbol{v}_h) \leq C_c^{s,*}\|\boldsymbol{u}\|_{V_h^{s,*}}\|\boldsymbol{v}_h\|_{V_h^s},
    \qquad
    a_h(\boldsymbol{u},\boldsymbol{v}_h)\leq C_c^*\|\boldsymbol{u}\|_{\boldsymbol{Z}_h^*}\|\boldsymbol{v}_h\|_{Z_h},
  \end{equation*}
  for all
  $\boldsymbol{u} \in \boldsymbol{Z}_h+(\widetilde{V}\times
  \bar{\widetilde{V}}\times \widetilde{V}^m)$ and
  $\boldsymbol{v}_h\in \boldsymbol{Z}_h$, with slightly different
  constants $C_c^{s,*}$ and $C_c^*$ due to the use of a continuous
  trace inequality instead of a discrete one. The dependence of the
  constants on the problem parameters, however, stays the same.
\end{remark}

To prove an inf-sup condition, we introduce the following
interpolation operators (see, e.g., \cite[Lemma 7]{Hansbo:2002} or
\cite[(2.5.30)]{Boffi:book}):
\begin{lemma}
  \label{lem:interpolation}
  There exist interpolation operators
  $\Pi_V: H_0(\mathrm{div},\Omega) \mapsto V_h \cap
  H_0(\mathrm{div},\Omega)$,
  $\Pi_{V}^m: H_0(\mathrm{div}, \Omega^d) \mapsto V_h^m \cap
  H_0(\mathrm{div},\Omega^d)$ such that for all
  $u\in [H^{k+1}(K)]^{\mathrm{dim}}$, $K\in \mathcal{T}^d$, and
  $u^m\in [H^{k+1}(K)]^{\mathrm{dim}}$, $K\in \mathcal{T}^d$ the
  following hold:
  \begin{enumerate}
  \item $( q,\nabla \cdot (u-\Pi_Vu))_K =0$ for all $q\in P_{k-1}(K)$,
    $K\in \mathcal{T}$.
  \item $( q^m,\nabla \cdot (u^m-\Pi_{V}^mu^m))_K=0$ for all
    $q\in P_{k-1}(K)$, $K\in \mathcal{T}^d$.
  \item $\langle \bar{q},(u-\Pi_V u)\cdot n \rangle_F=0$ for all
    $\bar{q}\in P_k(F)$ , $F\in \mathcal{F}$.
  \item $\langle \bar{q}^m,(u^m-\Pi_V u^m)\cdot n \rangle_F=0$ for all
    $\bar{q}^m\in P_k(F)$ , $F\in \mathcal{F}^d$.
  \item $\|u-\Pi_Vu\|_{p,K} \leq C h_K^{l-p}|u|_{l,K}$ with $p=0,1,2$
    and $\max(1,p)\leq l \leq k+1$, $K\in \mathcal{T}$.
  \item $\|u^m-\Pi_{V}^mu^m\|_{p,K} \leq C h_K^{l-p}|u^m|_{l,K}$ with
    $p=0,1,2$ and $\max(1,p)\leq l \leq k+1$, $K\in \mathcal{T}^d$.
  \item
    $\|\nabla \cdot (u-\Pi_Vu)\|_K \leq C h_K^{l}|\nabla \cdot
    u|_{\ell,K}$ with $l \leq k+1$, $K\in \mathcal{T}$.
  \item
    $\|\nabla \cdot (u^m-\Pi_{V}^mu^m)\|_K \leq C h_K^{l}|\nabla \cdot
    u^m|_{\ell,K}$ with $l \leq k+1$, $K\in \mathcal{T}^d$.
  \end{enumerate}
\end{lemma}
Furthermore, we denote the $L^2$-projection onto $\bar{V}_h$ by
$\bar{\Pi}_V$. For $v\in [H^l(K)]^{\mathrm{dim}}$, $1\leq l \leq k+1$,
we have:
\begin{subequations}
  \label{eq:barPiVestimate}
  \begin{align}
    \|v-\Bar{\Pi}_V v\|_{\partial K}
    &\leq C h_K^{l-1/2}\|v\|_{H^l(K)},
    \\
    \|\Pi_V v-\Bar{\Pi}_V v\|_{\partial K}
    &\leq C h_K^{l-1/2}\|v\|_{H^l(K)}.
  \end{align}
\end{subequations}
Let us define the space
\begin{equation*}
    \boldsymbol{Z}_h^0 := \big\{\boldsymbol{v}_h\in \boldsymbol{Z}_h:
    b_h((0,\bar{q}_h^s,\bar{q}_h^d,0,\bar{q}_h^m),\boldsymbol{v}_h) = 0
    \quad
    \forall (\bar{q}_h^s,\bar{q}_h^d,\bar{q}_h^m)\in \bar{Q}_h^s\times\bar{Q}_h^d\times \bar{Q}_h^m \big\}.      
\end{equation*}
We now prove the following inf-sup condition.
\begin{lemma}
  \label{lem:infsup cond of b_1}
  There exists a constant $C > 0$, independent of $h$, such that for
  any $(q_h,q_h^m)\in Q_h\times Q_h^m$,
  \begin{equation}
    \label{eq:b1infsup}
    C (\|q_h\|_{\Omega}^2+\|q_h^m\|^2_{\Omega^d})^{1/2}
    \leq \sup_{\boldsymbol{v}_h\in {\boldsymbol{Z}_h^0},\boldsymbol{v}_h\neq 0}
    -\frac{\sum_{j=s,d}( q_h,\nabla \cdot v_h)_{\Omega^j}+( q_h^m,\nabla \cdot v_h^m)_{\Omega^d}}{\|\boldsymbol{v}_h\|_{Z_h}}.
  \end{equation}
\end{lemma}
\begin{proof}
  Let $q_h\in Q_h$ and $q_h^m\in Q_h^m$. Then, since
  $q_h\in L^2_0(\Omega)$ and $q_h^m\in L_0^2(\Omega^d)$, by the
  standard inf-sup condition, there exist
  $v \in [H^1_0(\Omega)]^{\text{dim}}$ and
  $v^m\in [H_0^1(\Omega^d)]^{\text{dim}}$ such that:
  \begin{align*}
    -\nabla\cdot v=&q_h \text{ in }\Omega,& C\|v\|_{H^1(\Omega)}&\leq \|q_h\|_{\Omega},
    \\
    -\nabla\cdot v^m=&q_h^m \text{ in }\Omega^d,& C\|v^m\|_{H^1(\Omega^d)}&\leq \|q_h^m\|_{\Omega^d}.
  \end{align*}
  By \Cref{lem:interpolation}, \cref{eq:barPiVestimate}, and
  \cite[(1.24)]{Girault:2009},
  \begin{align*}
    \|(\Pi_V v,\bar{\Pi}_V v,\Pi_{V}^mv^m)\|^2_{Z_h}=
    &\|\epsilon(\Pi_V v)\|^2_{\Omega^s}+\sum_{K\in \mathcal{T}^s} h_K^{-1}\|\Pi_Vv-\bar{\Pi}_V v\|_{\partial K}^2
    \\
    &+\|\Pi_V v\|^2_{H({\rm div};\Omega^d)}+\|\Pi_{V}^mv^m\|^2_{H({\rm div};\Omega^d)}+\|\Bar{\Pi}_V (v^s)^t\|_{\Gamma^I}^2
    \\
    \leq
    & C(\|v\|^2_{1,\Omega^s}+\|v^m\|^2_{1,\Omega^d}+\|v\|^2_{1,\Omega^d}).
  \end{align*}
  Observe also that
  $(\Pi_V v,\bar{\Pi}_V v,\Pi_{V}^mv^m)\in \boldsymbol{Z}_h^0$ by the
  fact that $\bar{q}_h^j$, $j=s,d,m$ is single-valued and the
  properties of $\Pi_V$ and $\bar{\Pi}_V$. Therefore,
  \begin{align*}
    \sup_{\boldsymbol{v}_h\in \boldsymbol{Z}_h^0,\boldsymbol{v}_h\neq \boldsymbol{0}}
    -&\frac{\sum_{j=s,d} ( q_h,\nabla \cdot v_h)_{\Omega^j}+( q_h^m,\nabla \cdot v_h^m)_{\Omega^d}}{\|\boldsymbol{v}_h\|_{Z_h}}
    \\
     &\geq  \dfrac{\sum_{i=s,d} ( q_h,\nabla \cdot \Pi_V v)_{\Omega^j}+( q_h^m,\nabla \cdot \Pi_{V}^mv^m)_{\Omega^d}}
       {\|(\Pi_V v,\Bar{\Pi}_V v ,\Pi_{V}^m v^m)\|_{Z_h}}
    \\
     &\geq C \dfrac{\|q_h\|^2_{\Omega}+\|q_h^m\|_{\Omega^d}^2}{(\|v\|^2_{1,\Omega^s}+\|v\|^2_{1,\Omega^d}+\|v^m\|^2_{1,\Omega^d})^{1/2}}
    \\
     &\geq C (\|q_h\|^2_{\Omega}+\|q_h^m\|^2_{\Omega^d})^{1/2}.
  \end{align*}
\end{proof}
\begin{lemma}
 \label{lem:infsup cond of b2}
 There exists a constant $C > 0$, independent of $h$, such that for
 any
 $(\bar{q}_h^s,\bar{q}_h^d,\bar{q}_h^m)\in \bar{Q}_h^s\times
 \bar{Q}_h^d \times \bar{Q}_h^d$,
 \begin{multline}
   \label{eq:b2infsup}
   C\Big(\sum_{j=s,d}\sum_{K\in\mathcal{T}^j}h_K\|\bar{q}_h^j\|^2_{\partial K}+\sum_{K\in\mathcal{T}^d}h_K\|\bar{q}_h^m\|^2_{\partial K}\Big)^{1/2}
   \\
   \leq \sup_{\boldsymbol{v}_h \in \boldsymbol{Z}_h,\boldsymbol{v}_h\neq \boldsymbol{0}}
   \frac{\sum_{j=s,d}(\langle \bar{q}_h^j, v_h\cdot n \rangle_{\partial \mathcal{T}^j}
     +b_h^{I,j}(\bar{q}_h^j,\bar{v}_h))+\langle \bar{q}_h^m, v_h^m\cdot n \rangle_{\partial \mathcal{T}^d}}
   {\|\boldsymbol{v}_h\|_{Z_h}}.
 \end{multline}
\end{lemma}
\begin{proof}
  The proof is similar to that of \cite[Lemma 3]{Rhebergen:2018b}. We
  start by introducing an operator \cite[Proposition 2.10]{Du:book} to
  lift $\bar{q}_h^m\in \Bar{Q}_h^m$
    to $\Omega^d$ and $\bar{q}_h^j\in \Bar{Q}_h^j$ to $\Omega^j$,
    $j=s,d$. Let
  $R_k(\partial K):=\{\Bar{q}\in L^2(\partial K): \Bar{q}\in P_k(F),\
  \forall F \subset\partial K\}$ and let
  $L:R_k(\partial K) \mapsto [P_k(K)]^{\rm dim}$ be the BDM local
  lifting operator that satisfies for all
  $\bar{q}_h\in R_k(\partial K)$:
  \begin{subequations}
    \label{eq:lift}
    \begin{align}
      (L\Bar{q}_h)\cdot n&=h_K \Bar{q}_h \text{ on } \partial K,
      & \|L\Bar{q}_h\|_K &\leq C_0 h_K^{3/2}\|\Bar{q}_h\|_{\partial K},
      \\
      \|\nabla (L\Bar{q}_h)\|_{K}&\leq C_0h_K^{1/2} \|\Bar{q}_h\|_{\partial K},
      & \| L\Bar{q}_h\|_{\partial K}&\leq C_0 h_K \|\Bar{q}_h\|_{\partial K}. 
    \end{align}    
  \end{subequations}
  Here the constant $C_0 \ge 1$ only depends on the shape regularity
  of the mesh and the polynomial degree $k$. Using the same argument
  as in \cite[Lemma 6]{Cesmelioglu:2020}, we define
  \begin{equation*}
    L \bar{q}_h=
    \begin{cases}
      L \Bar{q}_h^s  & \forall K \in \mathcal{T}^s,\\
      L \Bar{q}_h^d  & \forall K \in \mathcal{T}^d. 
    \end{cases}
  \end{equation*}
  Then, $(L\bar{q}_h, 0, L\bar{q}_h^m)\in \boldsymbol{Z}_h$ and
  \begin{multline*}
    \sum_{j=s,d}(\langle \bar{q}_h^j, L \bar{q}_h\cdot n \rangle_{\partial \mathcal{T}^j}
    +b_h^{I,j}(\bar{q}_h^j,0))+\langle \bar{q}_h^m, L\bar{q}_h^m\cdot n \rangle_{\partial \mathcal{T}^d}
    \\
    = \sum_{j=s,d}\sum_{K\in\mathcal{T}^j}h_K\|\bar{q}_h^j\|^2_{\partial K}+\sum_{K\in\mathcal{T}^d}h_K\|\bar{q}_h^m\|^2_{\partial K}.
  \end{multline*}
  Furthermore, by \cref{eq:lift},
  \begin{align*}
    \|(L\Bar{q}_h,0,L\Bar{q}_h^m)\|_{Z_h}^2
    =& \|\epsilon(L\bar{q}_h^s)\|^2_{\Omega^s} + \sum_{K\in \mathcal{T}^s}h_K^{-1}\|L\bar{q}_h^s\|_{\partial K}^2
     \\
     &+\|L\bar{q}_h^d\|_{H(\rm div,\Omega^d)}^2+\|L\bar{q}_h^m\|_{H(\rm div,\Omega^d)}^2
    \\
    \leq& C_0^2 \Big(\sum_{j=s,d}\sum_{K\in \mathcal{T}^j}h_K\|\bar{q}_h^j\|_{\partial K}^2
      +\sum_{K\in \mathcal{T}^d}h_K\|\bar{q}_h^m\|_{\partial K}^2\Big).
  \end{align*}
  Therefore,
  \begin{align*}
    \sup_{\boldsymbol{v}_h \in \boldsymbol{Z}_h,\boldsymbol{v}_h\neq \boldsymbol{0}}
    & \frac{\sum_{j=s,d}(\langle \bar{q}_h^j, v_h\cdot n \rangle_{\partial \mathcal{T}^j}
      +b_h^{I,j}(\bar{q}_h^j,\bar{v}_h))+\langle \bar{q}_h^m, v_h^m\cdot n \rangle_{\partial \mathcal{T}^d}}{\|\boldsymbol{v}_h\|_{Z_h}}
    \\ 
    & \geq \frac{\sum_{j=s,d}\langle \bar{q}_h^j,L \Bar{q}_h^j \cdot n \rangle_{\partial \mathcal{T}^j}
      +\langle \bar{q}_h^m, L \Bar{q}_h^m\cdot n \rangle_{\partial \mathcal{T}^d}}{\|(L\Bar{q}_h,0,L\Bar{q}_h^m)\|_{Z_h}} 
    \\ 
    & \geq \frac{\sum_{j=s,d} \sum_{K\in \mathcal{T}^j}h_K\|\Bar{q}_h^j\|^2_{\partial K}
      +\sum_{K\in \mathcal{T}^d}h_K\|\Bar{q}_h^m\|^2_{\partial K}}
      {C_0\Big(\sum_{j=s,d}\sum_{K\in \mathcal{T}^j}h_K\|\bar{q}_h^j\|_{\partial K}^2
      +\sum_{K\in \mathcal{T}^d}h_K\|\bar{q}_h^m\|_{\partial K}^2\Big)^{1/2}}
    \\
    &= (C_0)^{-1}\Big(\sum_{j=s,d}\sum_{K\in\mathcal{T}^j}h_K\|\bar{q}_h^j\|^2_{\partial K}
      +\sum_{K\in\mathcal{T}^d}h_K\|\bar{q}_h^m\|^2_{\partial K}\Big)^{1/2}.
  \end{align*}
  The result follows with $C = (C_0)^{-1}$.
\end{proof}

The previous two lemmas are now used to prove the following main
inf-sup condition.

\begin{theorem}
  \label{thm:infsup condition b}
  There exists a constant $\beta_p>0$, independent of h, such that for
  all $\boldsymbol{p}_h \in \boldsymbol{P}_h$,
  \begin{equation}
    \label{eq:Stability of b}
    \beta_p\|\boldsymbol{p}_h\|_{P_h}
    \leq \sup_{\boldsymbol{v}_h\in \boldsymbol{Z}_h,\boldsymbol{v}_h\neq \boldsymbol{0}}
    \frac{\boldsymbol{b}_h(\boldsymbol{p}_h,\boldsymbol{v}_h)}{\|\boldsymbol{v}_h\|_{Z_h}}.
  \end{equation}
\end{theorem}
\begin{proof}
  \Cref{eq:b1infsup,eq:b2infsup} are equivalent to \cref{eq:Stability
    of b}, see \cite[Theorem 3.1]{Howell:2011}.
\end{proof}

We end this section by proving well-posedness of the HDG method
\cref{eq:HDG-dualporosityproblem-18}.

\begin{theorem}
  If $\beta>\beta_0$, then the discrete problem
  \cref{eq:HDG-dualporosityproblem-18} is well-posed.
\end{theorem}
\begin{proof}
  It is sufficient to show uniqueness. Letting $f=0$ and $g=0$ and
  choosing $\boldsymbol{v}_h=\boldsymbol{u}_h$ and
  $\boldsymbol{q}_h=-\boldsymbol{p}_h$ in
  \cref{eq:HDG-dualporosityproblem-18}, we obtain:
  \begin{equation*}
    0=a_h(\boldsymbol{u}_h,\boldsymbol{u}_h)+c(\boldsymbol{p}_h,\boldsymbol{p}_h).
  \end{equation*}
  By \Cref{lem:Coercivity a},
  \begin{equation*}
    C_e(\|\boldsymbol{u}_h\|_{V_h^s}^2
    +\|\bar{u}_h^t\|_{\Gamma^I}+\|u_h\|^2_{\Omega^d}+\|u_h^m\|^2_{\Omega^d}) + \sigma \kappa_m\|p_h^m-p_h\|_{\Omega^d}^2\leq 0.    
  \end{equation*}
  Therefore $p_h^m=p_h$ in $\Omega^d$ and
  $\boldsymbol{u}_h=\boldsymbol{0}$. Substituting these values in
  \cref{eq:HDG-dualporosityproblem-18}, we obtain:
  \begin{equation*}
    \boldsymbol{b}_h(\boldsymbol{p}_h,\boldsymbol{v}_h)=0
    \quad \forall \boldsymbol{v}_h\in \boldsymbol{Z}_h.    
  \end{equation*}
  Therefore, $\boldsymbol{p}_h=\boldsymbol{0}$ by \Cref{thm:infsup
    condition b}, concluding the proof.
\end{proof}

%---------------------------------------------------------------------
\section{Error Analysis}
\label{sec:error}

In this section, we present an a priori error analysis of the HDG
method in \cref{eq:HDG-dualporosityproblem-18}. For the analysis we
will use the BDM interpolation operator $\Pi_V$ as defined in
\Cref{lem:interpolation} and the $L^2$-projection operators
$\bar{\Pi}_V$ onto $\bar{V}_h$, $\Pi_Q$ onto $Q_h$, $\Pi_Q^m$ onto
$Q_h^m$, and $\bar{\Pi}_Q^j$ onto $\bar{Q}_h^j$, $j=s,d$. We have the
following standard estimates for $k\geq 0$ and $0\leq l \leq k$:
\begin{subequations}
  \label{eq:interpolation-est}
  \begin{align}
    \|q-\Pi_Q q\|_K & \leq Ch^l_K\|q\|_{l,K}&& \forall q \in H^l(K),
    \\
    \|q-\bar{\Pi}^d_Q q\|_{\partial K} & \leq Ch^{l+1/2}_K\|q\|_{l+1,K} && \forall q \in H^{l+1}(K).
  \end{align}
\end{subequations}
We define
$\boldsymbol{\Pi}u:=(\Pi_V {u},\bar{\Pi}_V u^s,\Pi_{V}^m{u^m})$,
$\boldsymbol{\Pi}p:=(\Pi_Q p,\bar\Pi_Q^s {p}^s,\bar\Pi_Q^d {p}^d,\Pi
_Q^m p^m,\bar\Pi_Q^d {p}^m)$ and introduce the following notation for
the errors:
\begin{align*}
  e^I_{u}&=u-\Pi_V u,  &  e^h_{u}&=u_h-\Pi_Vu 
  , &
  e^I_{u^m}&=u^m-\Pi_{V}^m u^m,  &  e^h_{u^m}&=u_h^m-\Pi_{V}^m u^m,
  \\
  e^I_{p}&=p-\Pi_Q p, & e^h_{p}&=p_h-\Pi_Q p ,
  &
  e^I_{p^m}&=p^m-\Pi_{Q}^mp^m,  &  e^h_{p^m}&=p_h^m-\Pi_{Q}^m p^m,
\end{align*}
and
\begin{align*}
  \bar{e}^I_{u^s}&=u^s|_{\Gamma_0^s}-\bar{\Pi}_V u^s,&     \bar{e}^h_{u^s}&=\bar{u}_h^s-\bar{\Pi}_V u^s,
  \\
  \bar{e}^I_{p^s} & = p^s|_{\Gamma_0^s}-\bar{\Pi}_Q^s p^s,  &\bar{e}^h_{p^s} & =p_h^s-\bar{\Pi}_Q^s p^s, 
  \\
  \bar{e}^I_{p^j}&=p^j|_{\Gamma_0^j}-\bar{\Pi}_Q^d p^j, & \bar{e}^h_{p^j}&=p^j_h-\bar{\Pi}_Q^d p^j, \quad j=d,m.
\end{align*}
We use the following compact notation:
\begin{align*}
  \boldsymbol{e}^I_u &=(e^I_{u},\bar e^I_{u},e^I_{u^m}),
  &
    \boldsymbol{e}^h_u&=(e^h_{u},\bar e^h_{u},e^h_{u^m}), 
  \\
  \boldsymbol{e}_p^I&=(e^I_{p},\bar{e}^I_{p^s},\bar{e}^I_{p^d},e^I_{p_m},\bar{e}^I_{p^m}),
  &
    \boldsymbol{e}_p^h&=(e^h_{p},\bar{e}^h_{p^s},\bar{e}^h_{p^d},e^h_{p_m},\bar{e}^h_{p^m}),
  \\
  \boldsymbol{e}^I_{p^j} &= (e^I_{p^j},\bar{e}^I_{p^j}),
  & \boldsymbol{e}^h_{p^j}&=(e^h_{p^j},\bar{e}^h_{p^j}),
\end{align*}
where $j=s,d,m$. From \cite[Lemmas 7 and 8]{Cesmelioglu:2020},
\begin{equation}
  \label{lem:estimates on xi u}
  \|\boldsymbol{e}^I_{u}\|_{V_h^{s,*}}\leq C h^{l-1} \|u\|_{l,\Omega^s}, \quad 2\leq l\leq k+1,
\end{equation}
and
\begin{equation}
  \label{lem:estimates on xi p}
  \|\boldsymbol{e}^I_{p^s}\|_{Q_h^j}\leq C h^l \|p^s\|_{l,\Omega^s},\quad 
  \|\boldsymbol{e}^I_{p^j}\|_{Q_h^j}\leq C h^l\|p^j\|_{l,\Omega^d},\quad  0\leq l\leq k, \quad j=d,m.
\end{equation}
The following lemma presents the error equations used to obtain our
error estimates in \Cref{thm:error} and \Cref{cor:apriori}.

\begin{lemma}[Error equation]
    \label{lem:reduced-error-equations}
    For any
    $(\boldsymbol{v}_h,\boldsymbol{q}_h) \in \boldsymbol{Z}_h \times
    \boldsymbol{P}_h$, the following holds: 
    \begin{subequations}
      \label{eq:reduced-Error-eqs}
      \begin{align}
        \label{eq:reduced-Error-eq-1}
        a_h(\boldsymbol{e}^h_u,\boldsymbol{v}_h)+ b_h(\boldsymbol{e}^h_p,\boldsymbol{v}_h)
        &=a_h^s(\boldsymbol{e}_{u}^I, \boldsymbol{v}_h^s) + a_h^d(e^I_{u},v_h^d)+a_h^m(e^I_{u^m},v_h^m),
        \\
        \label{eq:reduced-Error-eq-2}
        b_h(\boldsymbol{q}_h,\boldsymbol{e}^h_u)-c_h(\boldsymbol{e}^h_p,\boldsymbol{q}_h)
        &=0.
      \end{align} 
    \end{subequations}
\end{lemma}
\begin{proof}
  By \Cref{lem:consistency} and \cref{eq:HDG-dualporosityproblem-18},
  \begin{subequations}
    \begin{align}
      \label{eq:erreq-a}
      a_h(\boldsymbol{u},\boldsymbol{v}_h)+b_h(\boldsymbol{p},\boldsymbol{v}_h)
      &=a_h(\boldsymbol{u}_h,\boldsymbol{v}_h)+b_h(\boldsymbol{p}_h,\boldsymbol{v}_h),
      \\
      \label{eq:erreq-b}
      b_h(\boldsymbol{q}_h,\boldsymbol{u})-c_h(\boldsymbol{p},\boldsymbol{q}_h)
      &=b_h(\boldsymbol{q}_h,\boldsymbol{u}_h)-c_h(\boldsymbol{p}_h,\boldsymbol{q}_h),      
    \end{align}
  \end{subequations}
  for all $\boldsymbol{v}_h\in \boldsymbol{V}_h$,
  $\boldsymbol{q}_h\in \boldsymbol{P}_h$. Subtracting
  $a_h(\boldsymbol{\Pi}u,\boldsymbol{v}_h)+b_h(\boldsymbol{\Pi}p,\boldsymbol{v}_h)$
  from both sides of \cref{eq:erreq-a} and
  $b_h(\boldsymbol{q}_h,\boldsymbol{\Pi}u)-c_h(\boldsymbol{\Pi}p,\boldsymbol{q}_h)$
  from both sides of \cref{eq:erreq-b}, we find
  \begin{align*}
    a_h(\boldsymbol{e}_u^h,\boldsymbol{v}_h)+b_h(\boldsymbol{e}_p^h,\boldsymbol{v}_h)
    &= a_h(\boldsymbol{e}_u^I,\boldsymbol{v}_h)+b_h(\boldsymbol{e}_p^I,\boldsymbol{v}_h),
    \\
    b_h(\boldsymbol{q}_h,\boldsymbol{e}_u^h)-c_h(\boldsymbol{e}_p^h,\boldsymbol{q}_h)
    &= b_h(\boldsymbol{q}_h,\boldsymbol{e}_u^I)-c_h(\boldsymbol{e}_p^I,\boldsymbol{q}_h).
  \end{align*}
  Expanding the right hand sides:
  \begin{align*}
    a_h&(\boldsymbol{e}_u^I,\boldsymbol{v}_h)+b_h(\boldsymbol{e}_p^I,\boldsymbol{v}_h)
    = a_h^s(\boldsymbol{e}_u^I,\boldsymbol{v}_h) + a_h^I(\bar{e}_u^I,\bar{v}_h)+a_h^d(e_u^I, v_h) + a_h^m(e_{u^m}^I, v_h^m) 
    \\
     & + b_h^s(\boldsymbol{e}_{p^s}^I,v_h) + b_h^d(\boldsymbol{e}_{p^d}^I,v_h) + b_h^d(\boldsymbol{e}_{p^m}^I,v_h^m)
       + b_h^{I,s}(\bar{e}_{p^s}^I, \bar{v}_h)+b_h^{I,d}(\bar{e}_{p^d}^I, \bar{v}_h),
    \\
    b_h&(\boldsymbol{q}_h,\boldsymbol{e}_u^I)-c_h(\boldsymbol{e}_p^I,\boldsymbol{q}_h)
    = b_h^s(\boldsymbol{q}_h^s,e_u^I) + b_h^d(\boldsymbol{q}_h^d,e_u^I) + b_h^d(\boldsymbol{q}_h^m,e_{u^m}^I)
       \\
       &+ b_h^{I,s}(\bar{q}_h^s, \bar{e}_{u}^I)+b_h^{I,d}(\bar{q}_h^d, \bar{e}_{u}^I)-c_h(\boldsymbol{e}_p^I,\boldsymbol{q}_h).
  \end{align*}
  Since $\Pi_Q^j$, $\bar{\Pi}_Q^j$, $j=d,m$ are $L^2$-projections, we
  note that
  \begin{align*}
    b_h^s(\boldsymbol{e}^I_{p^s},{v}_h)&=0, \quad  b_h^d(\boldsymbol{e}^I_{p^d},{v}_h)=0, \quad
    b_h^d(\boldsymbol{e}^I_{p^m},{v}_h^m)=0,
    \\
    \quad b_h^{I,d}(\bar{e}_{p^d}^I, \bar{v}_h)&=0, \quad
    b_h^{I,s}(\bar{e}_{p^s}^I, \bar{v}_h)=0, \quad c(\boldsymbol{e}^I_p,\boldsymbol{q}_h)=0.    
  \end{align*}
  Furthermore, by \Cref{lem:interpolation},
  \begin{equation*}
    b_h^s(\boldsymbol{q}^s_h,{e}^I_{u})=0,
    \quad b_h^d(\boldsymbol{q}^d_h,{e}^I_{u})=0,
    \quad b_h^d(\boldsymbol{q}^m_h,e^I_{u^m})=0,
  \end{equation*}
  and by the definition of the $L^2$-projection $\bar{\Pi}_V$, we have
  \begin{equation*}
    b_h^{I,s}(\bar{q}_h^s,\bar{e}_{u}^I)=0,
    \quad b_h^{I,d}(\bar{q}_h^d,\bar{e}_{u}^I)=0,
    \quad a_h^I(\bar{e}_u^I, \bar{v}_h)=0.
  \end{equation*}
  The conclusion follows from the above identities.
\end{proof}
\begin{theorem}
 \label{thm:error}
 Let
 $(\boldsymbol{u},\boldsymbol{p})\in \boldsymbol{Z}\times
 \boldsymbol{P}$ be the solution of the dual-porosity-Stokes problem
 \cref{eq:Stokes,eq:dualporosity,eq:IC} such that
 $u^s\in [H^{k+1}(\Omega^s)]^{\rm dim}$,
 $u^d \in [H^{k}(\Omega^d)]^{\rm dim}$,
 $u^m \in [H^{k}(\Omega^d)]^{\rm dim}$, $p^j\in H^k(\Omega^j)$,
 $j=s,d$, $p^m\in H^k(\Omega^d)$, $k\geq 1$. Let
 $(\boldsymbol{u}_h, \boldsymbol{p}_h)\in \boldsymbol{Z}_h\times
 \boldsymbol{P}_h$ be the solution to
 \cref{eq:HDG-dualporosityproblem-18}. Then,
 \begin{subequations}
   \begin{align}
     \label{eq:e_u^hbnd-new-ac}
     \|\boldsymbol{e}^h_u\|_{Z_h}
     &\leq  CC_e^{-1} C_m h^k(\|u\|_{k+1,\Omega^s}+\|u\|_{k,\Omega^d}+\|u^m\|_{k,\Omega^d}),
     \\
     \label{eq:e_p^hbnd-new-ac}
     \|\boldsymbol{e}_p^h\|_{P_h}
     &\leq C \beta_p^{-1}(C_cC_e^{-1}+1)C_mh^k(\|u\|_{k+1,\Omega^s}+\|u\|_{k,\Omega^d}+\|u^m\|_{k,\Omega^d}),
   \end{align}   
 \end{subequations}
 where $C_e$ is the ellipticity constant in \Cref{lem:Coercivity a}, $\beta_p$ is the inf-sup constant in \Cref{thm:infsup condition b}, $C_c$ is the boundedness constant in \cref{eq:ahboundedness},
 $C_m = \max(C_c^{s,*},\kappa_f^{-1},\kappa_m^{-1})$ in which $C_c^{s,*}$ is the boundedness constant in \Cref{rem:extendedboundedness}, and $C$ is a
 generic constant independent of $h, \sigma, \mu, \kappa_f$, and
 $\kappa_m$.
\end{theorem}
\begin{proof}
  Choose
  $(\boldsymbol{v}_h,\boldsymbol{q}_h)=(\boldsymbol{e}_u^h,-\boldsymbol{e}_p^h)$
  in \cref{eq:reduced-Error-eqs}. Then, by \Cref{lem:Coercivity a} and
  positive semi-definiteness of $c_h$,
  \begin{equation}
    \label{eq:estimate-of-u}
    C_e
    \|\boldsymbol{e}^h_{{u}}\|_{Z_h}^2
    \leq a_h(\boldsymbol{e}^h_{{u}},\boldsymbol{e}^h_{{u}})
    + c_h(\boldsymbol{e}^h_{{p}},\boldsymbol{e}^h_{{p}})
    = a_h^s(\boldsymbol{e}_{u}^I, \boldsymbol{e}_{u}^h)
    + a_h^d(e^I_{u},e^h_{u}) + a_h^m(e^I_{u^m},e^h_{u^m}).
  \end{equation}
  We bound the right hand side of \cref{eq:estimate-of-u} by
  \Cref{rem:extendedboundedness}, the Cauchy--Schwarz inequality, and
  the definition of $\|\cdot\|_{V_h^{s,*}}$ as follows:
  \begin{equation}
    \label{eq:as}
    \begin{split}
      &a_h^s(\boldsymbol{e}_{u}^I, \boldsymbol{e}_{u}^h)+ a_h^d(e^I_{u},e^h_{u})+a_h^m(e^I_{u^m},e^h_{u^m})
      \\
      &\quad \leq C_c^{s,*}  \|\boldsymbol{e}_{u}^I\|_{V_h^{s,*}}\|\boldsymbol{e}_{u}^h\|_{V_h^s}
      + \kappa_f^{-1} \| e^I_{u}\|_{\Omega^d}\|e^h_{u}\|_{\Omega^d}+\kappa_m^{-1} \| e^I_{u^m}\|_{\Omega^d}\|e^h_{u^m}\|_{\Omega^d}
      \\
      &\quad \leq \max(C_c^{s,*},\kappa_f^{-1},\kappa_m^{-1}) (\|\boldsymbol{e}_{u}^I\|^2_{V_h^{s,*}}
      +\| e^I_{u}\|^2_{\Omega^d}+ \| e^I_{u^m}\|^2_{\Omega^d})^{1/2}
      \\
      &\quad\quad\times(\|\boldsymbol{e}_{u}^h\|^2_{V_h^{s,*}}+\|e^h_{u}\|^2_{\Omega^d}+\|e^h_{u^m}\|^2_{\Omega^d})^{1/2}
      \\
      &\quad \leq \max(C_c^{s,*},\kappa_f^{-1},\kappa_m^{-1}) (\|\boldsymbol{e}^I_{u}\|_{V_h^{s,*}}^2
      +\| e^I_{u}\|^2_{\Omega^d}+ \| e^I_{u^m}\|^2_{\Omega^d})^{1/2}
      \|\boldsymbol{e}^h_{{u}}\|_{Z_h}.      
    \end{split}
  \end{equation}
  \Cref{eq:e_u^hbnd-new-ac} now follows from
  \cref{eq:estimate-of-u,eq:as}, \Cref{lem:interpolation}, and
  \cref{eq:barPiVestimate}.

  We next prove \cref{eq:e_p^hbnd-new-ac}. By
  \cref{eq:reduced-Error-eq-1}, \Cref{lem:Continuity of a}, and
  following the same steps as in \cref{eq:as}, we have :
  \begin{multline*}
    |b_h(\boldsymbol{e}^h_p,\boldsymbol{v}_h)|
    \le
    |a_h(\boldsymbol{e}^h_u,\boldsymbol{v}_h)|
    +|a_h^s(\boldsymbol{e}_{u}^I, \boldsymbol{v}_h^s)|
    + |a_h^d(e^I_{u},v_h^d)| + |a_h^m(e^I_{u^m},v_h^m)|,
    \\
    \le (C_c\|\boldsymbol{e}_u^h\|_{Z_h}
    + \max(C_c^{s,*},\kappa_f^{-1},\kappa_m^{-1}) (\|\boldsymbol{e}^I_{u}\|_{V_h^{s,*}}^2
    +\| e^I_{u}\|^2_{\Omega^d}+ \| e^I_{u^m}\|^2_{\Omega^d})^{1/2})
    \|\boldsymbol{v}_h\|_{Z_h}.
  \end{multline*}
  Combining this with the inf-sup condition in \Cref{thm:infsup
    condition b}, \Cref{lem:interpolation}, \cref{eq:barPiVestimate},
  and \cref{eq:e_u^hbnd-new-ac},
  \begin{equation*}
    \begin{split}
      &\|\boldsymbol{e}^h_p\|_{P_h}
      \leq \beta_p^{-1}\sup_{\boldsymbol{v}_h\in \boldsymbol{Z}_h,\boldsymbol{v}_h\neq \boldsymbol{0}}
      \frac{\boldsymbol{b}_h(\boldsymbol{e}^h_p,\boldsymbol{v}_h)}{\|\boldsymbol{v}_h\|_{Z_h}}
      \\
      &\leq \beta_p^{-1}(C_c\|\boldsymbol{e}_u^h\|_{Z_h}
      + \max(C_c^{s,*},\kappa_f^{-1},\kappa_m^{-1}) (\|\boldsymbol{e}^I_{u}\|_{V_h^{s,*}}^2
      +\| e^I_{u}\|^2_{\Omega^d}+ \| e^I_{u^m}\|^2_{\Omega^d})^{1/2})
      \\
      & \le
      C \beta_p^{-1}(C_cC_e^{-1}+1)\max(C_c^{s,*},\kappa_f^{-1},\kappa_m^{-1})h^k(\|u\|_{k+1,\Omega^s}+\|u\|_{k,\Omega^d}+\|u^m\|_{k,\Omega^d}),
    \end{split}
  \end{equation*}
  which is the desired result.
\end{proof}

\begin{corollary}
  \label{cor:apriori} 
  Under the assumptions of \Cref{thm:error},
  {\fontsize{11.0}{11.0}\selectfont
  \begin{subequations}
    \begin{align}
      \label{eq:error-for-u}
     \|\boldsymbol{u}-\boldsymbol{u}_h\|_{Z_h}
        \leq &  C(C_e^{-1}C_m+1)h^{k} \del[0]{\|u\|_{k+1,\Omega^s}+\|u\|_{k+1,\Omega^d}+\|u^m\|_{k+1,\Omega^d}},
      \\
      \label{eq:error-for-p}
     \|\boldsymbol{p}-\boldsymbol{p}_h\|_{P_h}
        \leq  & Ch^k (\|p\|_{k,\Omega^s}+\|p\|_{k,\Omega^d}+\|p^m\|_{k,\Omega^d})
      \\ \nonumber
      &+ C\beta_p^{-1} (C_cC_e^{-1}+1)C_mh^k(\|u\|_{k+1,\Omega^s}+\|u\|_{k,\Omega^d}+\|u^m\|_{k,\Omega^d}),
    \end{align}    
  \end{subequations}
  }
  where $\boldsymbol{u}:=(u, u^s|_{\Gamma_0^s}, u^m)$,
  $\boldsymbol{p}:=(p, p^s|_{\Gamma_0^s}, p^d|_{\Gamma_0^d}, p^m,
  p^m|_{\Gamma_0^m})$, and $C$ is a constant a constant which depends on the shape regularity of meshes and the polynomial degree $k$ but is independent of
  $h, \sigma, \mu, \kappa_f$, and $\kappa_m$, and where
  \begin{equation*}
      \begin{split}
          C_m 
          &=
          \max(2\mu\max(1+C, \beta+C),\kappa_f^{-1},\kappa_m^{-1}),
          \\
          C_e 
          &=
          \min\{\mu(1-C^2/\beta)\min(1,\beta),\alpha
          \kappa_f^{-1/2},\kappa_f^{-1},\kappa_m^{-1}\},
          \\
          C_c 
          &=
          \max(2\mu\max(1+C_{\rm tr}, \beta+C_{\rm tr}),\kappa_f^{-1},\kappa_m^{-1},\alpha \mu
          \kappa_f^{-1/2}).
      \end{split}
  \end{equation*}
\end{corollary}
\begin{proof}
  We start by proving \cref{eq:error-for-u}. By a triangle inequality,
  \cref{eq:e_u^hbnd-new-ac}, \cref{lem:estimates on xi u}, and
  \Cref{lem:interpolation}, we obtain: 
  \begin{align*}
    \|\boldsymbol{u}&-\boldsymbol{u}_h\|_{Z_h}
    \le \|\boldsymbol{e}^h_u\|_{Z_h}+\|\boldsymbol{e}_u^I\|_{Z_h} 
    \\
    \leq &CC_e^{-1}C_mh^k(\|u\|_{k+1,\Omega^s}+\|u\|_{k,\Omega^d}+\|u^m\|_{k,\Omega^d}) 
    \\
       &\qquad+  Ch^k(\|u\|_{k+1,\Omega^s}+\|u\|_{k+1,\Omega^d}+\|u^m\|_{k+1,\Omega^d})
    \\
    \leq &C(C_e^{-1}C_m+1)h^{k}\big(\|u\|_{k+1,\Omega^s}+\|u\|_{k+1,\Omega^d}+\|u^m\|_{k+1,\Omega^d}\big).
  \end{align*}
  We next prove \cref{eq:error-for-p}. By the triangle inequality,
  \cref{lem:estimates on xi p}, and \cref{eq:e_p^hbnd-new-ac},
    \begin{align*}
    \|\boldsymbol{p}-\boldsymbol{p}_h\|_{P_h}  \leq& \|\boldsymbol{e}^I_p\|_{P_h}+\|\boldsymbol{e}^h_p\|_{P_h}
    \\
    \leq& Ch^k (\|p\|_{k,\Omega^s}+\|p\|_{k,\Omega^d}+\|p^m\|_{k,\Omega^d})
    \\
    &+ C\beta_p^{-1} (C_cC_e^{-1}+1)C_mh^k(\|u\|_{k+1,\Omega^s}+\|u\|_{k,\Omega^d}+\|u^m\|_{k,\Omega^d}).
  \end{align*}
\end{proof}

%---------------------------------------------------------------------
\section{Numerical examples}
\label{sec:Numerical ex}

The numerical examples in this section have been implemented using the
NGSolve library \cite{Schoberl:1997,Schoberl:2014}. The interior penalty parameter in \cref{eq:HDG-dualporosityproblem-18} is set to $\beta=10k^2$ for all examples which, by the explicit upper bound of $C_{\text{tr}}$ in \cite{Hestheven:2003}, is sufficient for stability if $k \le 5$. 

%---------------------------------------------------------------------
\subsection{Example 1}
\label{ex:1}

In our first numerical example, we verify our theoretical results
using a manufactured solution adapted from \cite[Section
6.1]{Cesmelioglu:2020}. Let $\Omega =(0,1)^2$,
$\Omega^s=(0,1)\times(0.5,1)$, and $\Omega^d=(0,1)\times(0,0.5)$. We
choose the following parameters: $\kappa_m=1$, $\kappa_f=1$, $\mu=1$,
$\sigma=1/2$, and
$\alpha=\mu\sqrt{\kappa_f}(1+4\pi^2)/2$. Boundary conditions and
source terms are set such that the exact solution to our problem is
given by:
\begin{align*}
  u^s
  &=
    \begin{bmatrix}
      -2\pi^{-2}\sin(\pi x) \exp(y/2)
      \\
      \pi^{-1}\cos(\pi x) \exp(y/2),
    \end{bmatrix},
  &
    p^s &= \frac{\kappa_f\mu -2}{\kappa_f\pi}\cos(\pi x) \exp(y/2),
  \\
  u^d
  &=
    \begin{bmatrix}
      -2\sin(\pi x) \exp(y/2),
      \\
      \pi^{-1}\cos(\pi x) \exp(y/2)
    \end{bmatrix},
  &
    p^d &= -\frac{2}{\kappa_f\pi}\cos(\pi x) \exp(y/2),
  \\
  u^m
  &=\begin{bmatrix}
    -\sin(\pi x) \cos(2 \pi y)
    \\
    -2\cos(\pi x)\sin(2 \pi y)
  \end{bmatrix},
  &
    p^m &= \frac{1}{\kappa_m\pi} \cos(\pi x) \cos(2\pi y).
\end{align*}
\Cref{tab:ratesconv-s} presents the $L^2$-errors and convergence rates
for $u$ and $p$ in $\Omega^s$ and \Cref{tab:ratesconv-d} presents the
$L^2$-errors and convergence rates for $u$ and $p$ in $\Omega^d$, and
$u^m$ and $p^m$ in $\Omega^d$ together with a column demonstrating
mass conservation.  These tables show that the rates of convergence of
$\norm[0]{u_h-u}_{\Omega^s}$, $\norm[0]{u_h-u}_{\Omega^d}$,
$\norm[0]{u_h^m-u^m}_{\Omega^d}$, $\norm[0]{p_h-p}_{\Omega^s}$,
$\norm[0]{p_h-p}_{\Omega^d}$, $\norm[0]{p_h^m-p^m}_{\Omega^d}$,
$\norm[0]{\nabla (u_h-u)}_{\Omega^s}$,
$\norm[0]{\nabla \cdot (u_h-u)}_{\Omega^d}$, and
$\norm[0]{\nabla \cdot (u_h^m - u^m)}_{\Omega^d}$ are at least $k$ or
higher, thereby corroborating \Cref{cor:apriori} and \cref{eq:Numerical properties f,eq:Numerical properties d,eq:Numerical properties e}. Note that the tables actually show that
$\norm[0]{u_h-u}_{\Omega^s}$, $\norm[0]{u_h-u}_{\Omega^d}$, and
$\norm[0]{u_h^m-u^m}_{\Omega^d}$ converge with an asymptotic rate of
convergence of $k+1$, even though this is not shown by our error
analysis.

\begin{table}[!ht]
  \centering{
    \begin{tabular}{cccccccc}
    \multicolumn{8}{l}{\textbf{In the Stokes region in $\Omega^s$}}\\
      \hline
      Cells
      & $\norm{e_u}_{\Omega^s}$ & Rate & $\norm{e_p}_{\Omega^s}$ 
      & Rate &$\norm{\nabla e_u}_{\Omega^s}$& Rate & $\norm{\nabla \cdot u_h}_{\Omega^s}$  \\
      \hline
      \multicolumn{8}{l}{$k=2$} \\
      32 & 5.4e-04 &- &3.4e-02&-& 1.8e-02&& 6.7e-17\\
    128 & 6.8e-05 & 3.0 &  6.4e-03 & 2.4 & 4.6e-03&2.0& 3.8e-17  \\
    512 & 8.6e-06 & 3.0&  1.3e-03 & 2.3  &1.1e-03&2.0& 4.6e-17 \\
    2048 & 1.1e-06 & 3.0 &  2.6e-04 & 2.3 & 2.8e-04 &2.0& 4.9e-17 \\
    8192 & 1.4e-07 & 3.0 & 5.9e-05 & 2.2 &7.1e-05&2.0& 4.8e-17 \\
      \multicolumn{8}{l}{$k=3$} \\
     32& 2.5e-05 &-& 1.7e-03&-&1.3e-03&-&5.8e-17\\
    128 & 1.6e-06 & 4.0 &  1.7e-04 & 3.3 & 1.5e-04& 3.0 & 5.7e-17 \\
    512 & 1.0e-07 & 4.0&  1.9e-05 & 3.2 &1.9e-05 &3.0& 5.2e-17 \\
    2048 & 6.4e-09 & 4.0 & 2.2e-06 & 3.1 &2.3e-06 & 3.0 & 5.1e-17 \\
    8192 & 4.0e-10 & 4.0 & 2.7e-07 & 3.0 &2.9e-07 & 3.0 & 5.2e-17\\
      \hline
    \end{tabular}
    \caption{Errors and rates of convergence, for the problem as set up
    in \Cref{ex:1}, in $\Omega^s$  for
    the velocity and pressure fields using polynomial degrees $k=2$ and
    $k=3$. Here $e_u:=u_h - u$, $e_p=p_h-p$, $e_u^m:=u_h^m-u^m$, $e_p^m = p_h^m-p^m$.}
\label{tab:ratesconv-s}
}
\end{table}
\begin{table}[!ht]
  \centering{
    \begin{tabular}{cccccccc}
       \multicolumn{8}{l}{\textbf{In microfractures in }$\Omega^d$}\\
      \hline
Cells & $\|e_u\|_{\Omega^d}$  & Rate &    $ \|e_p\|_{\Omega^d}$ & Rate & $ \|\nabla \cdot e_u\|_{\Omega^d}$ & Rate & $\|\Phi\|_{\Omega^d}$ \\
        \hline
     \multicolumn{8}{l}{$k=2$} \\
     32& 2.1e-03 &-&6.6e-03&-&6.3e-02&-&9.7e-08\\
   128 & 2.7e-04 & 3.0 & 1.6e-03 & 2.0 &1.6e-02 & 2.0 & 1.3e-09 \\
    512 & 3.4e-05 & 3.0 & 4.1e-04 & 2.0 & 4.0e-03 & 2.0 & 1.9e-11\\
    2048 & 4.2e-06 & 3.0  &  1.0e-04 & 2.0 & 9.9e-04 & 2.0& 2.9e-13\\
    8192 &  5.3e-07 & 3.0 & 2.6e-05 & 2.0 & 2.5e-04 & 2.0 & 4.6e-15 \\
      \multicolumn{8}{l}{$k=3$} \\
      32& 9.4e-05&-&4.2e-04&-&4.1e-03&-&1.3e-11\\
    128 & 5.8e-06 & 4.0 &  5.4e-05 & 3.0& 5.2e-04 & 3.0 & 1.3e-14 \\
    512 & 3.6e-07 & 4.0 & 6.7e-06 & 3.0& 6.5e-05 & 3.0  & 1.0e-15\\
    2048 & 2.3e-08 & 4.0& 8.4e-07 & 3.0 &  8.1e-06 & 3.0 & 9.7e-16\\
   8192 &1.4e-09 & 4.0 & 1.1e-07 & 3.0&  1.0e-06 & 3.0  & 9.6e-16\\
      \hline
       \multicolumn{8}{l}{\textbf{In the matrix in }$\Omega^d$}\\
           \hline
      Cells & $\|e_u^m\|_{\Omega^d}$  & Rate & $\|e_p^m\|_{\Omega^d}$  & Rate & $ \|\nabla \cdot e_u^m\|_{\Omega^d}$ & Rate & $\|\Phi^m\|_{\Omega^d}$ \\
        \hline
    \multicolumn{8}{l}{$k=2$} \\
    32& 1.6e-02 &-&1.0e-02&-& 6.3e-02&-&9.2e-06\\
    128  & 2.1e-03 & 2.9 & 2.7e-03 & 2.0& 1.6e-02 & 2.0  & 1.2e-07\\
    512 &  2.7e-04 & 3.0  & 2.7e-03 & 2.0 & 4.0e-03 & 2.0 & 1.7e-09\\
    2048 & 3.4e-05 & 3.0 & 1.7e-04 & 2.0 & 9.9e-04 & 2.0 & 2.6e-11\\
    8192 & 4.3e-06 & 3.0 & 4.2e-05 & 2.0 & 2.5e-04 & 2.0 & 4.0e-13\\
    \multicolumn{8}{l}{$k=3$} \\
    32& 1.8e-03 &-& 1.7e-03 &-&4.1e-03&-&1.2e-09\\
    128 & 1.2e-04 & 4.0 & 2.2e-04 & 3.0 & 5.2e-04 & 3.0 & 1.3e-12\\
    512 & 7.5e-06 & 4.0 & 2.8e-05 & 3.0 & 6.5e-05 & 3.0 & 2.2e-15\\
    2048 & 4.7e-07 & 4.0 & 3.5e-06 & 3.0 &  8.1e-06 & 3.0 & 1.7e-15\\
    8192 &2.9e-08 & 4.0 & 4.4e-07 & 3.0 & 1.0e-06 & 3.0 & 1.8e-15\\
    \end{tabular}
    \caption{Errors and rates of convergence, for the problem as set up
    in \Cref{ex:1}, in microfractures in
    $\Omega^d$ (top) and in the matrix in $\Omega^d$ (bottom) for
    the velocity and pressure fields using polynomial degrees $k=2$ and
    $k=3$. Here $e_u:=u_h - u$, $e_p=p_h-p$, $e_u^m:=u_h^m-u^m$, $e_p^m = p_h^m-p^m$,
    $\Phi:=\sigma \kappa_m(p_h^d-p_h^m)+\nabla \cdot u_h^d -\Pi_Q^d g$
    and $\Phi^m:=\sigma \kappa_m(p_h^m-p_h^d)+\nabla \cdot u_h^m$.}
\label{tab:ratesconv-d}
}
\end{table}

%---------------------------------------------------------------------
\subsection{Example 2}
\label{ss:ex2}

We next simulate fluid flow around wellbores with open-hole completion
in a naturally fractured reservoir. We present two cases: (i) a
vertical production wellbore; and (ii) a horizontal production
wellbore, both with open-hole completion \cite{Seale:2008}. The
examples presented below are inspired by \cite[Section
6.4]{Mahbub:2019b}, \cite[Section 6.2]{Mahbub:2019}, and
\cite[Examples 5.3, 5.4]{WEN:2022}.

%---------------------------------------------------------------------
\subsubsection{A vertical production wellbore}
\label{ex2-vertical}

For this example, we set $\Omega^s=(1/2,1)\times(1/2,3/2)$ and
$\Omega^d=(0,3/2)^2\setminus \overline{\Omega}^s$ and define the
boundaries and the interface as
\begin{align*}
    \Gamma^s &:= \cbr[0]{(x,y) :\, y=3/2\text{ and } x\in (1/2,1)},
    \\
    \Gamma^d &:= \cbr[0]{(x,y):\, x=0 \text{ or } y=0 \text{ or } x=3/2 \text{ or } y=3/2}\backslash \Gamma^s ,
    \\
    \end{align*}
    and
  \begin{align*}  
  	\Gamma^I &:= \cbr[0]{(x,y):\, x=1/2 \text{ or } y=1/2 \text{ or } x=1 },
  \end{align*}
on which we impose the following boundary conditions:
\begin{subequations}
  \label{eq:wellbore-bc}
  \begin{align}
    \label{eq:Ex2bc3}
    (-p^s I+\mu \epsilon( u^s))n&=0 && \text{ on } \Gamma^s,
    \\
    \label{eq:Ex2bc4}
    p^m&=5\times 10^4 &&\text{ on } \Gamma^d,
    \\
    p^d&=10^4 &&\text{ on } \Gamma^d.
  \end{align}  
\end{subequations}
See \Cref{fig:Domain Example2} for a depiction of the domain and its
boundaries. Note that \cref{eq:Ex2bc3} imposes an outflow boundary
condition for the free flow in $\Omega^s$. The parameters for this
problem are chosen as
\begin{equation}
  \label{eq:wellbore-parameters}
  h=1/64, \,
  \kappa_m=10^{-5}, \, 
  \kappa_f=10^{-1}, \,
  \mu=10^{-3}, \,
  \sigma=0.9, \,
  f=0, \,
  g=0.
\end{equation}
  
\begin{figure}[tbp]
  \centering
  \caption{A plot of the domain used in 
    \Cref{ex2-vertical}. }
  \begin{tikzpicture}[scale=1.3]
    \draw [color=black,fill = purple!10!pink](1,1) rectangle (2,3);
    \draw [ultra thin,color=black] (0,0)--(0,3) node[midway,left]{$\Gamma^d$};
    \draw [ultra thin,color=black] (0,3)--(1,3) node[midway,above]{$\Gamma^d$};
    \draw [ultra thin,color=black] (2,3)--(3,3) node[midway,above]{$\Gamma^d$};
    \draw [ultra thin,color=black] (3,3)--(3,0) node[midway,right]{$\Gamma^d$};
    \draw [ultra thin,color=black] (3,0)--(0,0) node[midway,below]{$\Gamma^d$};
    \draw [thin,color=blue] (1,3)--(2,3)
    node[ midway,above]{$\Gamma^s$};
    \draw [ultra thin,color=black] (2,1)--(1,1)
    node[midway,above]{$\Gamma^I$};
    \draw [ultra thin,color=black] (1,1)--(1,3)
    node[midway,left]{$\Gamma^I$};
    \draw [ultra thin,color=black] (2,1)--(2,3)
    node[midway,right]{$\Gamma^I$};
    \node at (1.5,2) {$\Omega^s$};
    \node at (1.5,0.5) {$\Omega^d$};
  \end{tikzpicture}
  \label{fig:Domain Example2}
\end{figure}
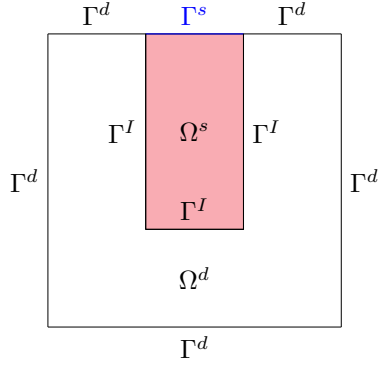

\Cref{fig:pressureex2} shows the pressure field in the wellbore and
microfractures while \Cref{fig:pressure in matrixex2} shows the
pressure field in the matrix.  The pressure difference between $p_h^m$
and $p_h$ in $\Omega^d$ results in fluid flow from the matrix to the
microfractures, while the pressure difference between $p_h$ in
$\Omega^d$ and $p_h$ in $\Omega^s$ drives the fluid from the
microfractures into the wellbore. The latter is observed in
\Cref{fig:velocityex2} in which we plot the streamlines and magnitude
of the velocity of the fluid flow in the microfractures and the
wellbore. We furthermore observe that once in the wellbore the
velocity magnitude of the fluid is significantly higher than in the
surrounding microfractures and that the fluid is driven towards the
outflow boundary $\Gamma^s$. We plot the streamlines and magnitude of
the velocity of the fluid in the matrix in \Cref{fig:velocity in
  matrixex2}. Here we observe that although fluid is driven towards
the wellbore, there is no fluid exchange between the matrix and the
wellbore as expected from the no-exchange interface condition on the
matrix velocity imposed by the model.

Moreover, we remark that mass is conserved pointwise on the elements
as predicted by \cref{eq:Numerical properties}. Indeed, we compute:
\begin{align*}
  \|\nabla \cdot u_h^s\|_{\Omega^s}
  &=1.0\cdot 10^{-11},
  \\
  \|\sigma \kappa_m (p_h^d-p_h^m)+\nabla \cdot u_h^d\|_{\Omega^d}
  &=2.4 \cdot 10^{-11},
  \\
  \|\sigma \kappa_m (p_h^m-p_h^d)+\nabla \cdot u_h^m\|_{\Omega^d}
  &=1.2 \cdot 10^{-16}.
\end{align*}

\begin{figure}[tbp]
  \centering
  \caption{The pressure and velocity solutions for the problem as set
    up in \Cref{ex2-vertical} with $\sigma = 0.9$. Left column:
    Pressure $p_h$ and velocity $u_h$ fields in $\Omega$. Right
    column: Pressure $p^m_h$ and velocity $u^m_h$ fields in the matrix
    in $\Omega^d$.  }
  \subfloat[Pressure $p_h$ in
  $\Omega$. \label{fig:pressureex2}]{\includegraphics[width=0.48\textwidth]{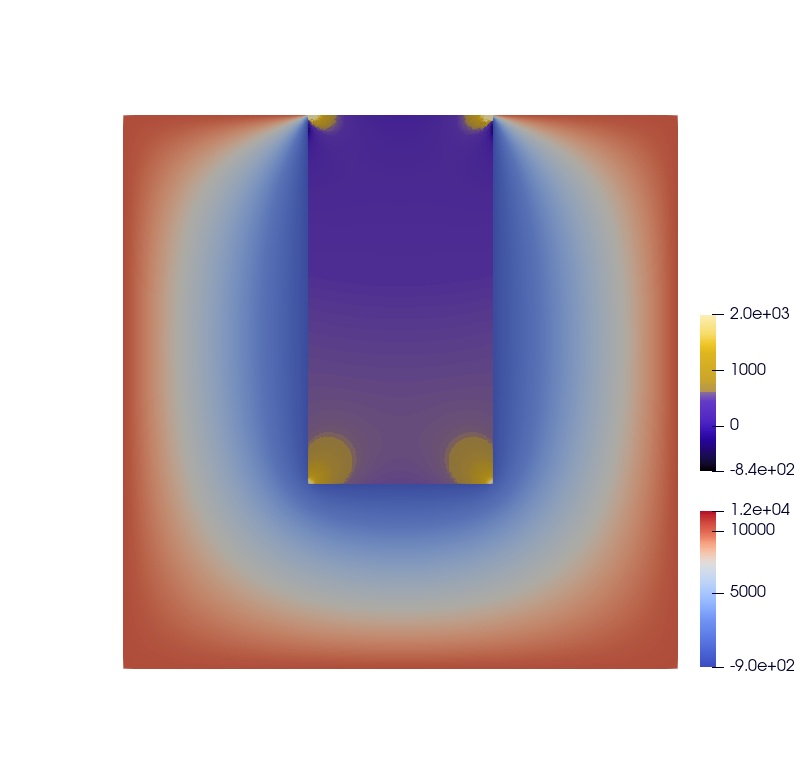}}
  \quad
  \subfloat[Pressure $p^m_h$ in the matrix in
  $\Omega^d$. \label{fig:pressure in matrixex2}]
  {\includegraphics[width=0.48\textwidth]{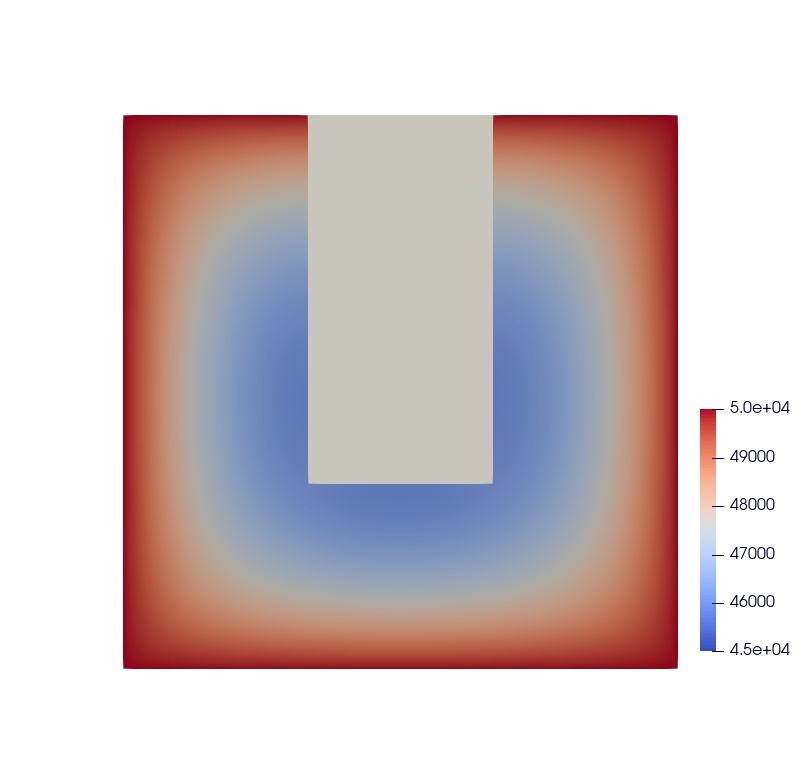}}
  \\
  \subfloat[Velocity field $u_h$ in
  $\Omega$. \label{fig:velocityex2}]{\includegraphics[width=0.48\textwidth]{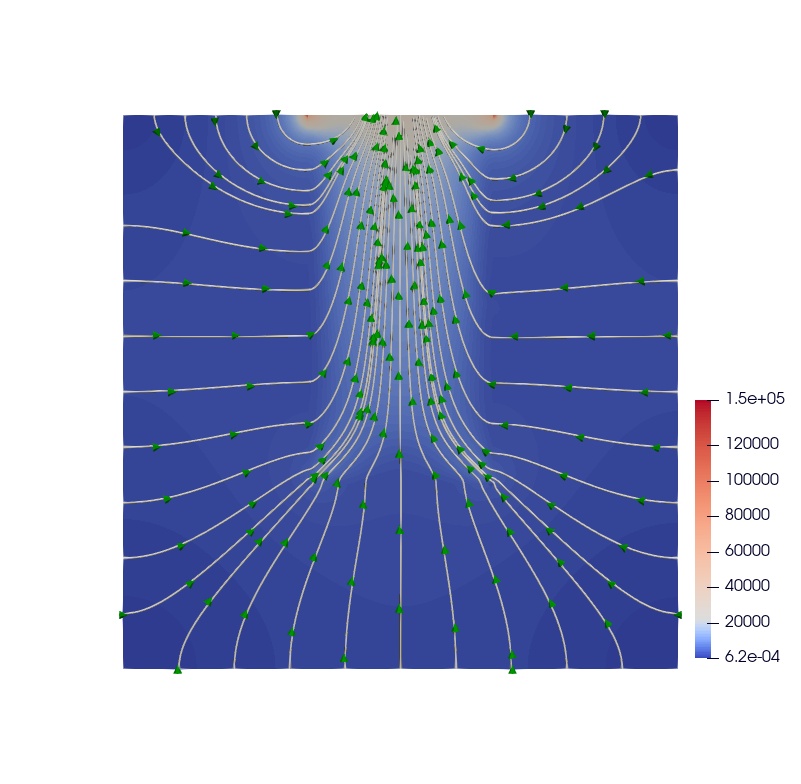}}
  \quad
  \subfloat[Velocity field $u^m_h$ in the matrix in
  $\Omega^d$. \label{fig:velocity in
    matrixex2}]{\includegraphics[width=0.48\textwidth]{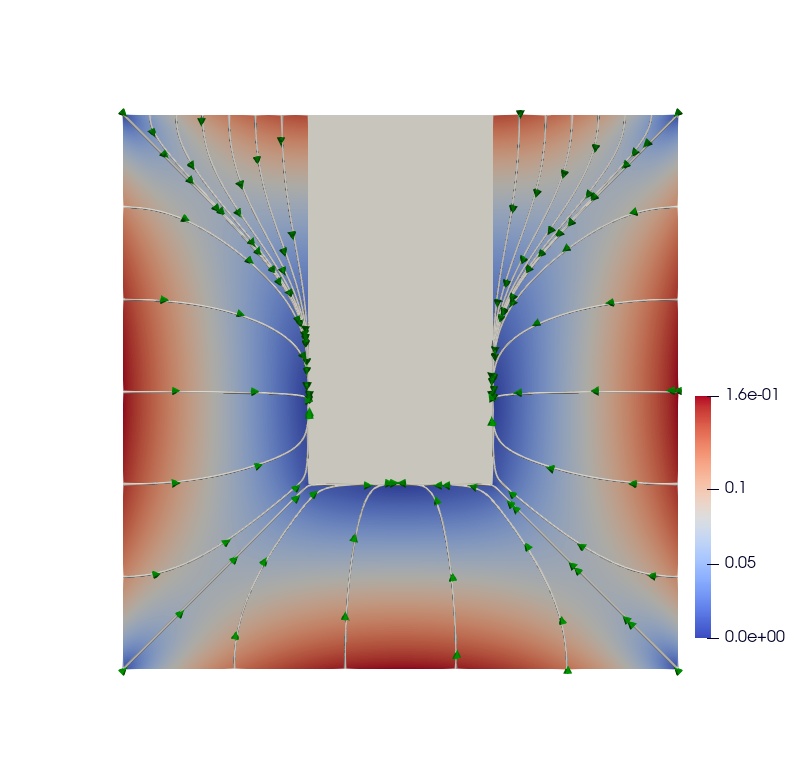}}
  \label{fig:Example2 figure}
\end{figure}

Finally, \Cref{fig: velocity and velocity in matrix field sigma=0.5,fig:
  velocity and velocity in matrix field sigma=0.1} show the velocity
fields obtained using shape factors $\sigma=0.5$ and $\sigma=0.1$,
respectively, with all other parameters the same as before. As
expected, the velocity magnitude decreases with decreasing $\sigma$.

\begin{figure}[tbp]
  \centering
  \caption{The velocity fields $u_h$ in $\Omega$ and $u_h^m$ in the
    matrix in $\Omega^d$ for the problem as set up in
    \Cref{ex2-vertical} with $\sigma=0.5$.}
  \subfloat[Velocity field $u_h$ in $\Omega$. \label{fig:velocityin
    matrix
    sigma05}]{\includegraphics[width=0.48\textwidth]{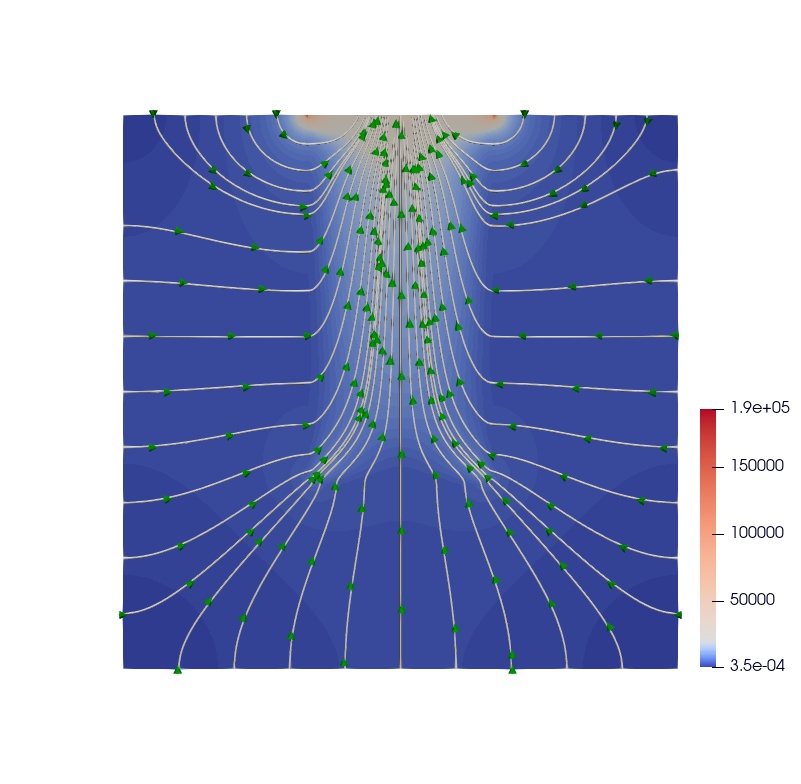}}
  \quad
  \subfloat[Velocity field $u_h^m$ in the matrix in
  $\Omega^d$.\label{fig:velocity in matrix
    sigma05}]{\includegraphics[width=0.48\textwidth]{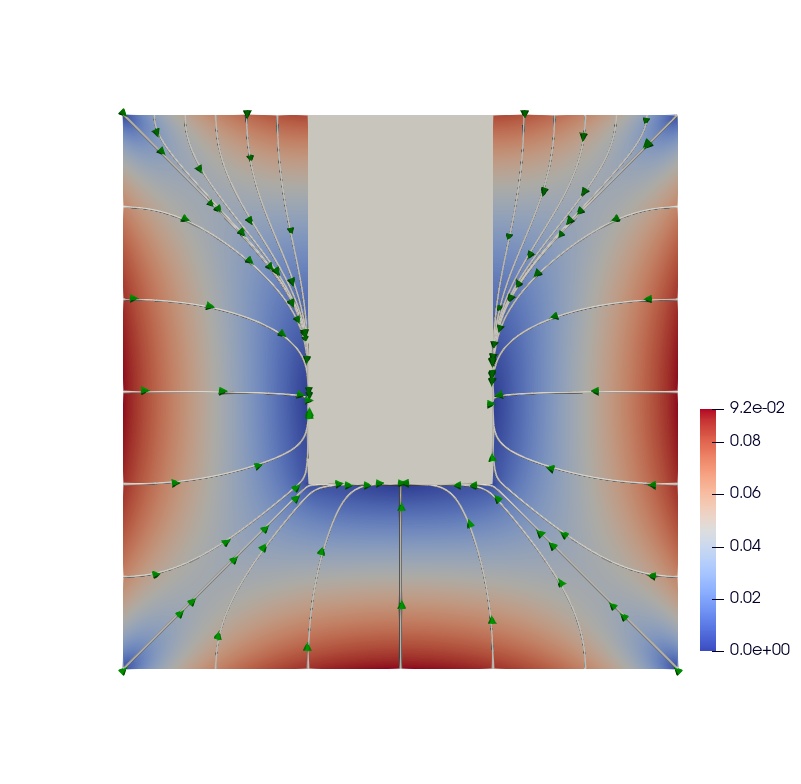}}
  \label{fig: velocity and velocity in matrix field  sigma=0.5}
\end{figure}

\begin{figure}[tbp]
  \centering
  \caption{The velocity fields $u_h$ in $\Omega$ and $u_h^m$ in the
    matrix in $\Omega^d$ for the problem as set up in
    \Cref{ex2-vertical} with $\sigma=0.1$.}
  \subfloat[Velocity field $u_h$ in $\Omega$. \label{fig:velocityin
    matrix
    sigma01}]{\includegraphics[width=0.48\textwidth]{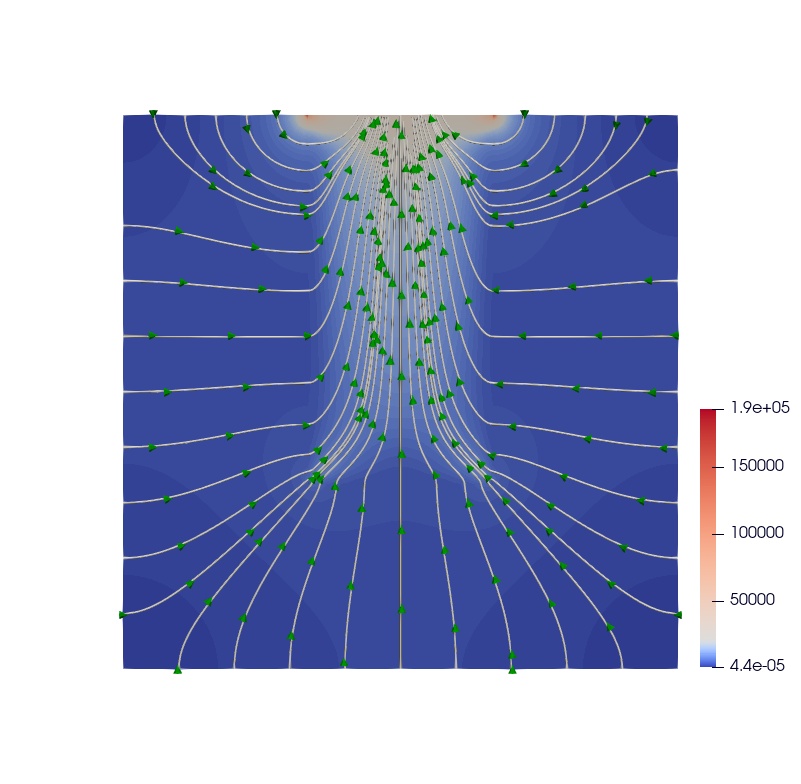}}
  \quad
  \subfloat[Velocity field $u_h^m$ in the matrix in
  $\Omega^d$. \label{fig:velocity in matrix
    sigma01}]{\includegraphics[width=0.48\textwidth]{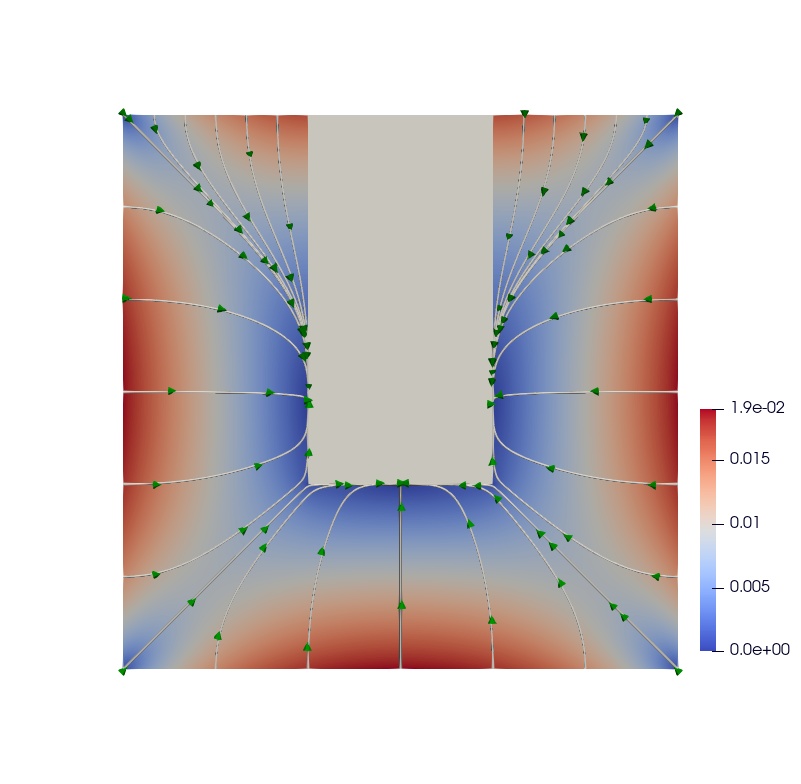}}
  \label{fig: velocity and velocity in matrix field  sigma=0.1}
\end{figure}

%---------------------------------------------------------------------
\subsubsection{Horizontal production wellbore}
\label{ex2-horizontal}

For this example, we change the domain such that
\begin{equation*}
  \Omega^s=(1/4,1)\times(1/4,1/2)\cup [1,5/4)\times (1/4,1/2]\cup (1,5/4)\times (1/2,3/2),
\end{equation*}
and $\Omega^d=(0,3/2)^2\setminus \overline{\Omega}^s$. The boundaries
are defined as
\begin{equation*}
  \begin{split}
  \Gamma^s &:= \cbr[0]{(x,y):\, y=3/2\text{ and } x\in (1,5/4)},
  \\
    \Gamma^d &:= \cbr[0]{(x,y):\, x=0 \text{ or } y=0 \text{ or } x=3/2 \text{ or } y=3/2}\setminus \Gamma^s ,
    \\
    \Gamma^I &:= \cbr[0]{(x,y):\, x=1/4 \text{ or } x=1 \text{ or } x=5/4\text{ or } y=1/4  \text{ or } y=1/2}.
  \end{split}
\end{equation*}
See \Cref{fig:Domain horizontal wellbore } for a depiction of the
domain and its boundaries. The parameters and boundary conditions are
as in \cref{eq:wellbore-parameters,eq:wellbore-bc}.
\begin{figure}[tbp]
  \centering
  \caption{A plot of the domain used in \Cref{ex2-horizontal,ss:ex3}.}   
  \begin{tikzpicture}[scale=2.8]
    \draw [color=purple!10!pink,fill = purple!10!pink](0.25,0.25) rectangle (1.25,0.5);
     \draw [color=purple!10!pink,fill = purple!10!pink](1,0.5) rectangle (1.25,1.5);
    \draw [ultra thin,color=black] (0,0)--(0,1.5) node[midway,left]{$\Gamma^d$};
    \draw [ultra thin,color=black] (0,1.5)--(1,1.5) node[midway,above]{$\Gamma^d$};
    \draw [ultra thin,color=black] (1.25,1.5)--(1.5,1.5) node[midway,above]{$\Gamma^d$};
    \draw [ultra thin,color=black] (1.5,1.5)--(1.5,0) node[midway,right]{$\Gamma^d$};
    \draw [ultra thin,color=black] (1.5,0)--(0,0) node[midway,below]{$\Gamma^d$};
    \draw [thin,color=blue] (1,1.5)--(1.25,1.5)
    node[ midway,above]{$\Gamma^s$};
    \draw [ultra thin,color=black] (0.25,0.25)--(1.25,0.25)
    node[midway,above]{$\Gamma^I$};
    \draw [ultra thin,color=black] (1.25,0.25)--(1.25,1.5)
    node[midway,right]{$\Gamma^I$};
    \draw [ultra thin,color=black] (1,1.5)--(1,0.5)
    node[midway,left]{$\Gamma^I$};
    \draw [ultra thin,color=black] (0.25,0.25)--(0.25,0.5)
    node[midway,left]{$\Gamma^I$};
    \draw [ultra thin,color=black] (0.25,0.5)--(1,0.5)
    node[midway,above]{$\Gamma^I$};
    \node at (1.125,1) {$\Omega^s$};
    \node at (0.75,0.125) {$\Omega^d$};
  \end{tikzpicture}
  \label{fig:Domain horizontal wellbore }
\end{figure}

The pressure field in the wellbore and microfractures is shown in
\Cref{fig:pressureexhorizontalwellbore} while the pressure field in
the matrix is shown in \Cref{fig:pressure in matrix horizontal
  wellbore}. The streamlines and the magnitude of the velocity of the
fluid flow in the microfractures and the wellbore and in the matrix
for $\sigma=0.9$, $0.5$, and $0.1$ are depicted in
\Cref{fig:velocityex horizontal wellbore,fig:velocity in matrix
  horizontal wellbore}, \Cref{fig:Example horizontal wellbore
  velocities sigma 05}, and \Cref{fig:Example horizontal wellbore
  velocities sigma 01}, respectively. As in \Cref{ex2-vertical}, the
difference between the matrix pressure and the pressure in the
microfractures drives the fluid from the matrix to the
microfractures. The difference between the pressure in the
microfractures and the pressure in the wellbore drives the fluid from
the microfractures to the wellbore. We also observe that the fluid
leaves the domain at $\Gamma^s$ with a velocity magnitude larger than
elsewhere in the domain and, as in \Cref{ex2-vertical}, that the
magnitude of the fluid velocity decreases as the shape parameter
$\sigma$ is decreased. As a final remark to this section, mass is
conserved pointwise on the elements with
\begin{align*}
  \|\nabla \cdot u_h^s\|_{\Omega^s}
  &=3.2\cdot 10^{-11},
  \\
  \|\sigma \kappa_m (p_h^d-p_h^m)+\nabla \cdot u_h^d\|_{\Omega^d}
  &=1.5 \cdot 10^{-10},
  \\
  \|\sigma \kappa_m (p_h^m-p_h^d)+\nabla \cdot u_h^m\|_{\Omega^d}
  &=6.7 \cdot 10^{-15}.
\end{align*} 

\begin{figure}[tbp]
  \centering
  \caption{The pressure and velocity solutions for the problem as set
    up in \Cref{ex2-horizontal} with $\sigma = 0.9$. Left column:
    Pressure $p_h$ and velocity $u_h$ fields in $\Omega$. Right
    column: Pressure $p^m_h$ and velocity $u^m_h$ fields in the matrix
    in $\Omega^d$.  }
  \subfloat[Pressure $p_h$ in
  $\Omega$. \label{fig:pressureexhorizontalwellbore}]{\includegraphics[width=0.48\textwidth]{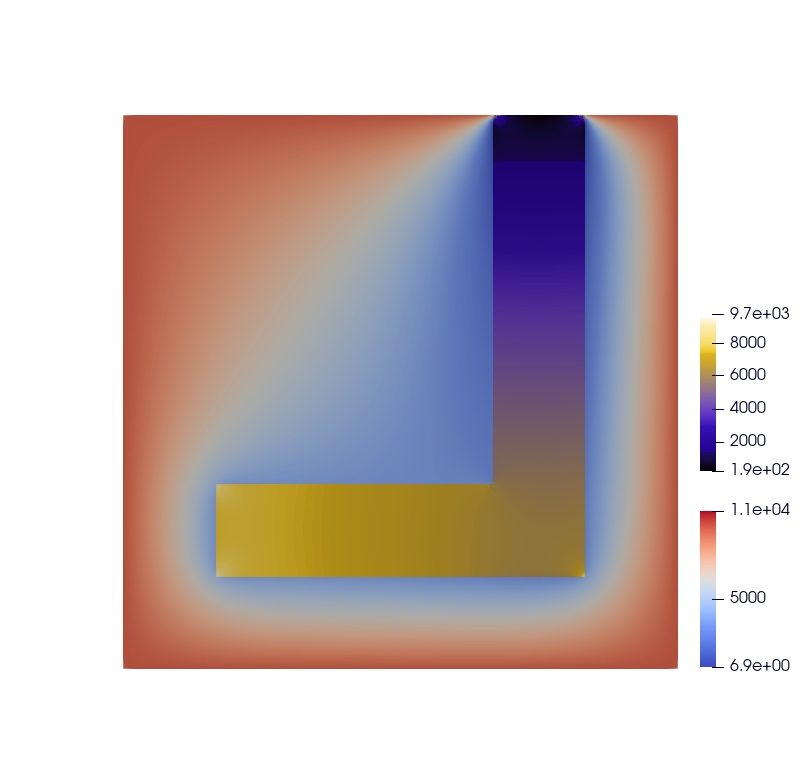}}
  \quad
  \subfloat[Pressure $p^m_h$ in the matrix in
  $\Omega^d$. \label{fig:pressure in matrix horizontal wellbore}]
  {\includegraphics[width=0.48\textwidth]{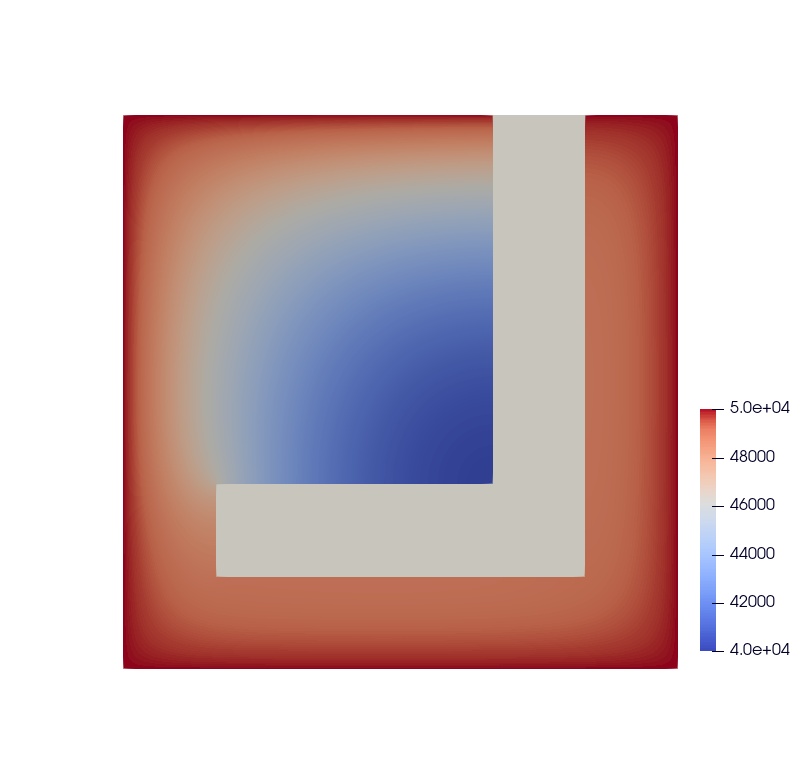}}
  \\
  \subfloat[Velocity field $u_h$ in
  $\Omega$. \label{fig:velocityex horizontal wellbore}]{\includegraphics[width=0.48\textwidth]{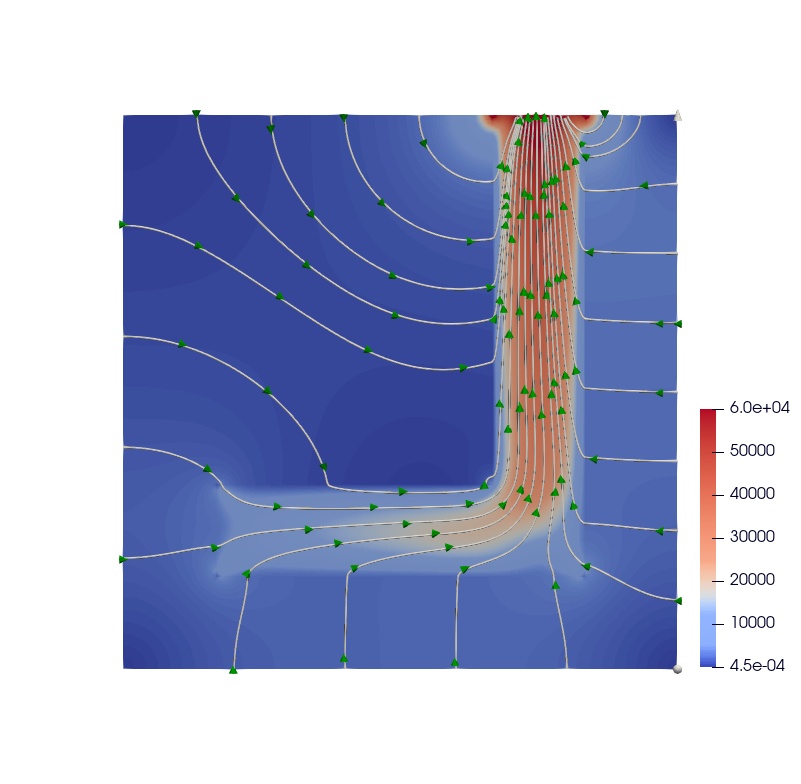}}
  \quad
  \subfloat[Velocity field $u^m_h$ in the matrix in
  $\Omega^d$. \label{fig:velocity in
    matrix horizontal wellbore}]{\includegraphics[width=0.48\textwidth]{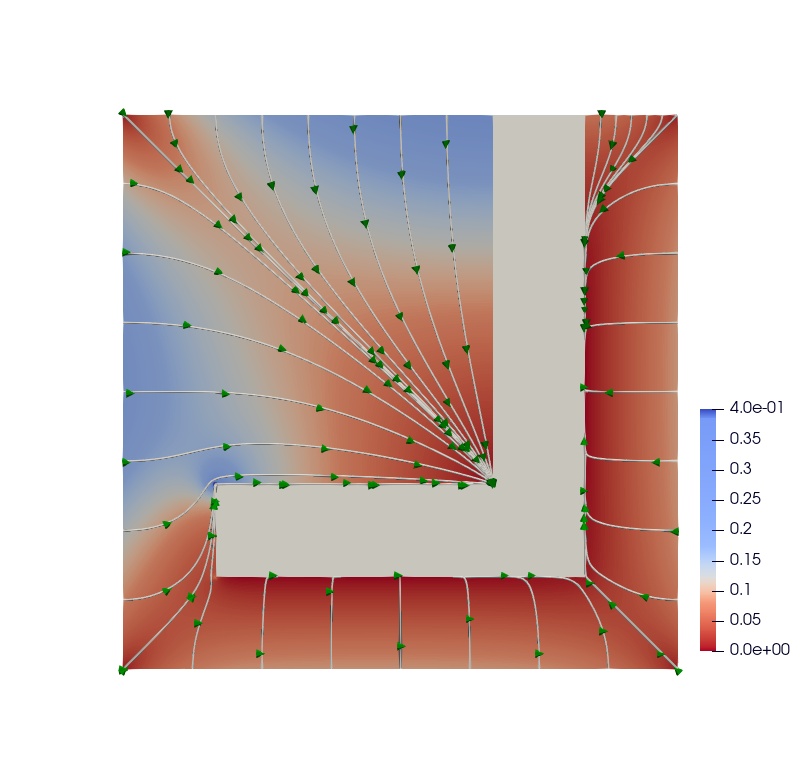}}
  \label{fig:Example horizontal wellbore figure}
\end{figure}

\begin{figure}[tbp]
  \centering
  \caption{The velocity solutions for the problem as set up in
    \Cref{ex2-horizontal} with $\sigma = 0.5$. Left column: Velocity
    $u_h$ field in $\Omega$. Right column: Velocity $u^m_h$ field in
    the matrix in $\Omega^d$. }
  \subfloat[Velocity field $u_h$ in $\Omega$. \label{fig:velocityex
    horizontal wellbore
    sigma05}]{\includegraphics[width=0.48\textwidth]{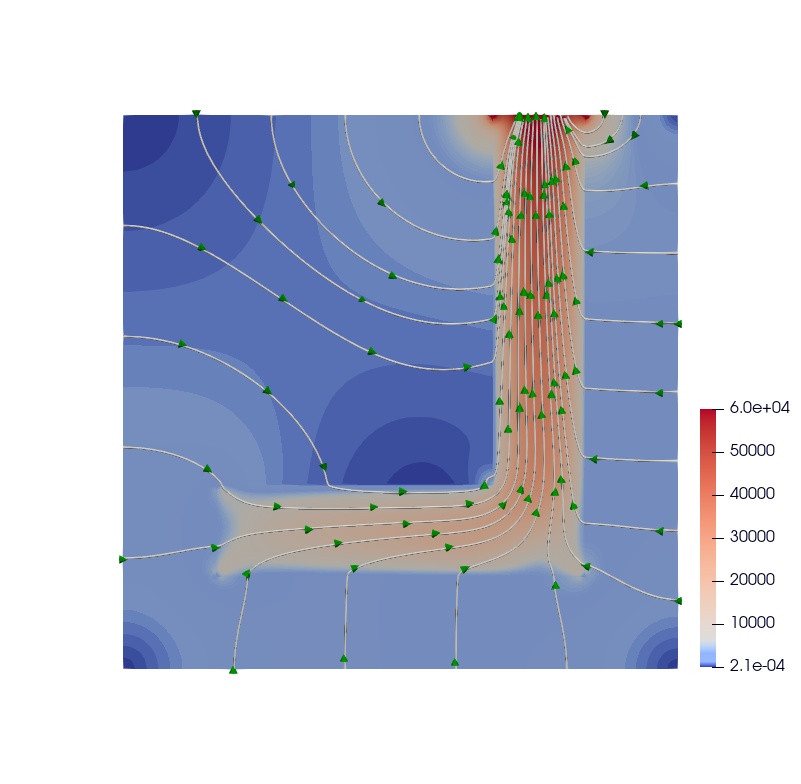}}
  \quad
  \subfloat[Velocity field $u^m_h$ in the matrix in
  $\Omega^d$. \label{fig:velocity in matrix horizontal wellbore sigma
    05}]{\includegraphics[width=0.48\textwidth]{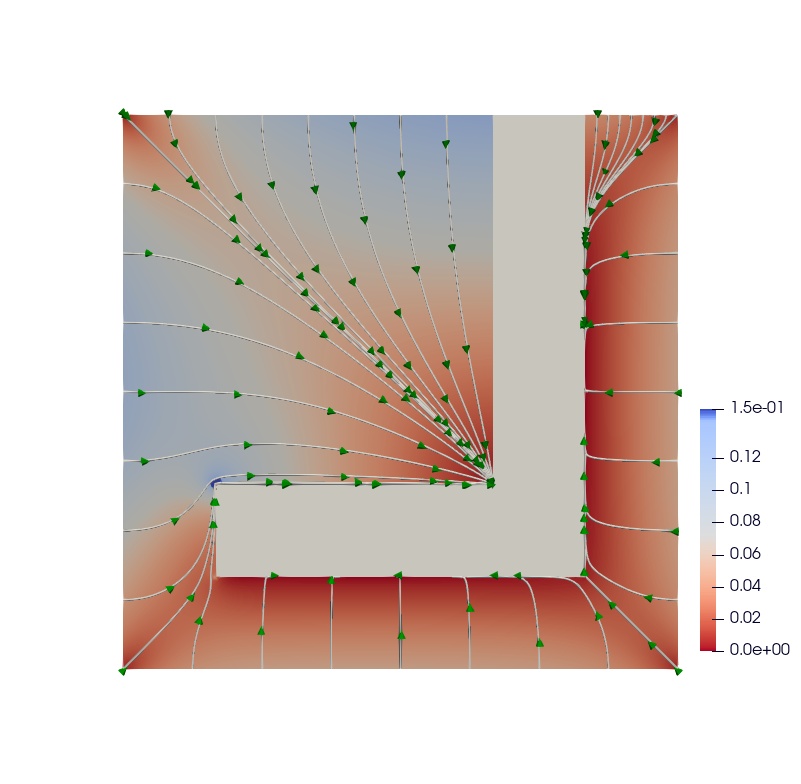}}
  \label{fig:Example horizontal wellbore velocities sigma 05}
\end{figure}

\begin{figure}[tbp]
  \centering
  \caption{The velocity solutions for the problem as set up in
    \Cref{ex2-horizontal} with $\sigma = 0.1$. Left column: Velocity
    $u_h$ field in $\Omega$. Right column: Velocity $u^m_h$ field in
    the matrix in $\Omega^d$.}
  \subfloat[Velocity field $u_h$ in $\Omega$. \label{fig:velocityex
    horizontal wellbore
    sigma01}]{\includegraphics[width=0.48\textwidth]{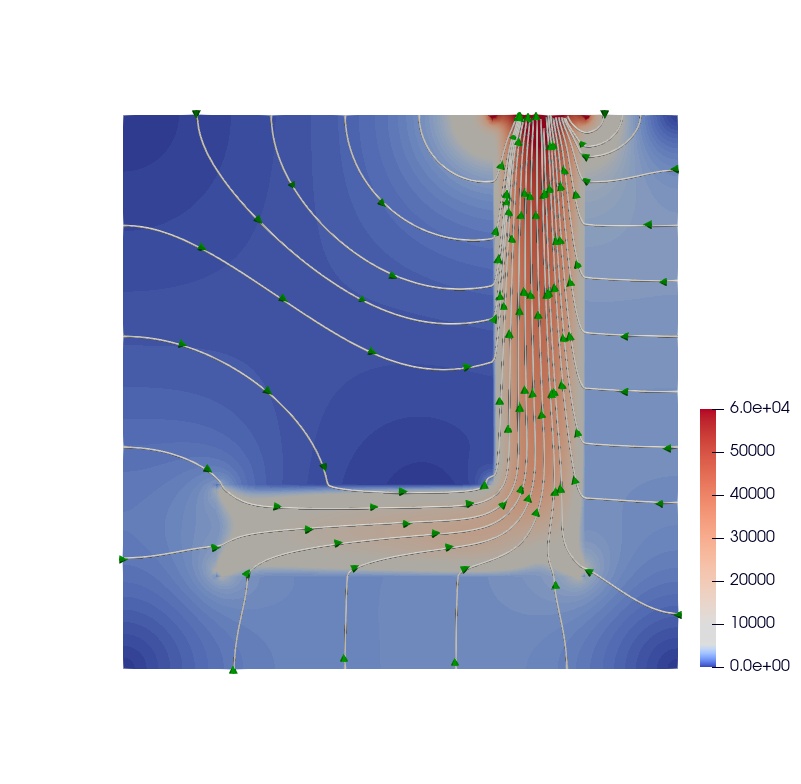}}
  \quad
  \subfloat[Velocity field $u^m_h$ in the matrix in
  $\Omega^d$. \label{fig:velocity in matrix horizontal wellbore sigma
    01}]{\includegraphics[width=0.48\textwidth]{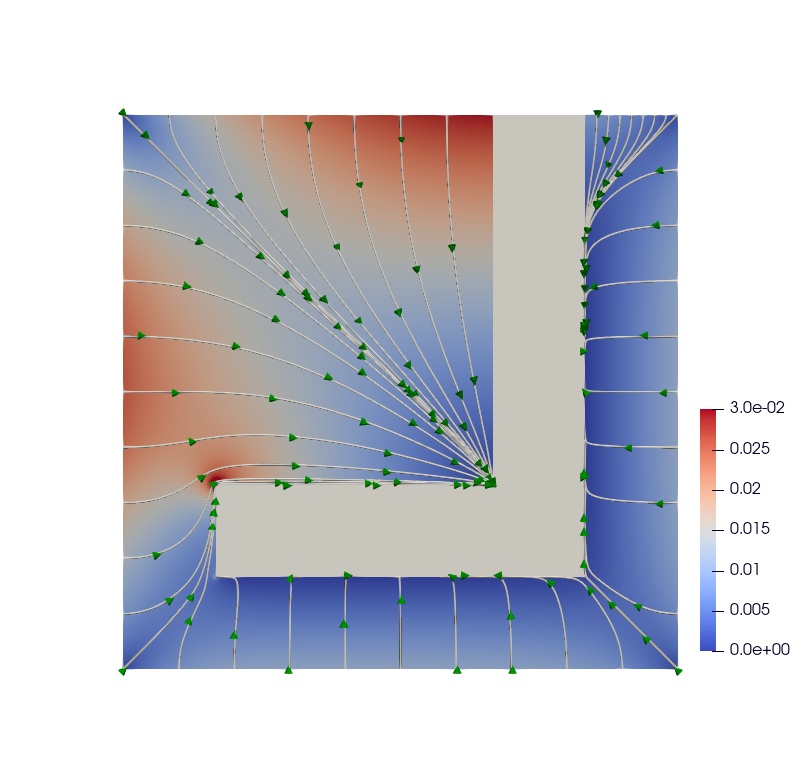}}
  \label{fig:Example horizontal wellbore velocities sigma 01}
\end{figure}

%---------------------------------------------------------------------
\subsection{Example 3}
\label{ss:ex3}

In this last example, we show that the discretization also performs
well on discontinuous data. We consider the same setup as in
\Cref{ex2-horizontal} changing only the permeabilities: on each
element in the mesh, $\kappa_f$ and $\kappa_m$ are now randomly
distributed constants such that $\kappa_f\in [10^{-2},1]$ and
$\kappa_m\in [10^{-6},10^{-4}]$ (see \Cref{fig:example3 horiz
  permeability} for a plot of the permeabilities).

Our results in \Cref{fig:Example3 horizontal wellbore figure} shows
the dependence of the flow on the permeability: the fluid follows a
non-uniform flow pattern in the dual-porosity region $\Omega^d$ as
opposed to the uniform flow field observed in \Cref{fig:Example
  horizontal wellbore figure}. The fluid flows from $\Omega^d$ into
the wellbore region $\Omega^s$, avoiding low permeability regions in
$\Omega^d$, but flowing freely in regions with high permeability. As
in \Cref{ex2-horizontal}, the fluid once again leaves the domain
through $\Gamma^s$. Finally, as shown by the results below, mass is
also conserved pointwise on the elements when dealing with
discontinuous permeabilities:
\begin{align*}
  \label{eq:masscons-ex3-hor}
  \|\nabla \cdot u_h^s\|_{\Omega^s}
  &=3.9\cdot 10^{-11},
  \\
  \|\sigma \kappa_m (p_h^d-p_h^m)+\nabla \cdot u_h^d\|_{\Omega^d}
  &=1.7 \cdot 10^{-10},
  \\
  \|\sigma \kappa_m (p_h^m-p_h^d)+\nabla \cdot u_h^m\|_{\Omega^d}
  &=7.7 \cdot 10^{-15}.
\end{align*} 

\begin{figure}[tbp]
    \centering
    \caption{The permeabilities used for the problem as described in
      \Cref{ss:ex3}. Left, the random permeability field
      $\kappa_f \in [10^{-2},1]$ in the microfractures, and right, the
      random permeability field $\kappa_m \in [10^{-6},10^{-4}]$ in
      the matrix.}
    \subfloat[Permeability $\kappa_f$ in
    microfractures. \label{fig:permeability}]{\includegraphics[width=0.48\textwidth]{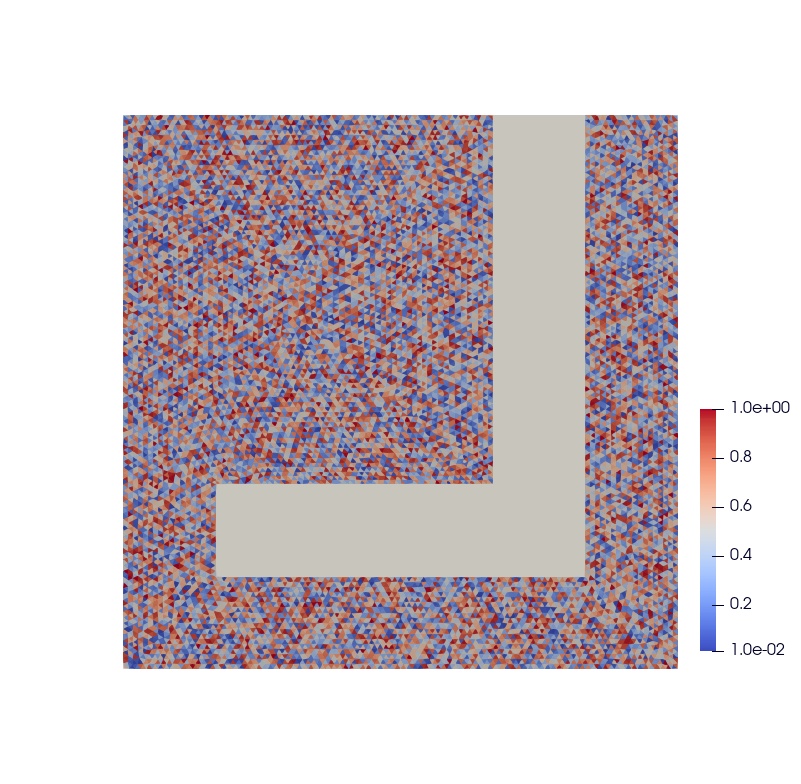}}
    \quad
    \subfloat[Permeability $\kappa_m$ in the
    matrix. \label{fig:velocity in
      matrix2}]{\includegraphics[width=0.48\textwidth]{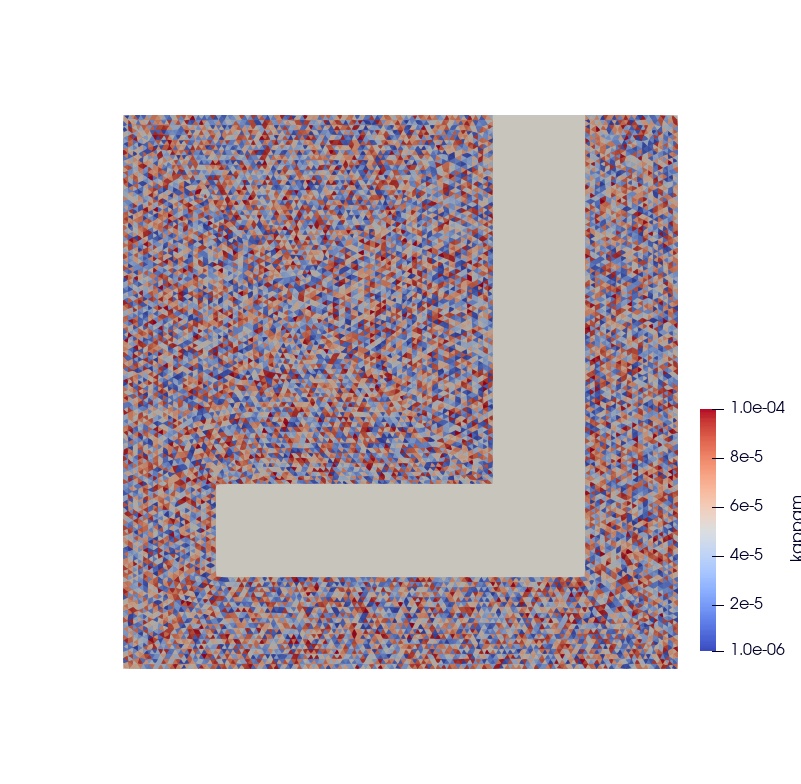}}
    \label{fig:example3 horiz permeability}
\end{figure}

\begin{figure}[tbp]
  \centering
  \caption{The pressure and velocity solutions for the problem as set
    up in \Cref{ss:ex3}.  Left column: Pressure $p_h$ and velocity
    $u_h$ fields in $\Omega$. Right column: Pressure $p_h^m$ and
    velocity $u_h^m$ fields in the matrix in $\Omega^d$. }
  \subfloat[Pressure $p_h$ in $\Omega$. \label{fig:horiz
    pressure3}]{\includegraphics[width=0.48\textwidth]{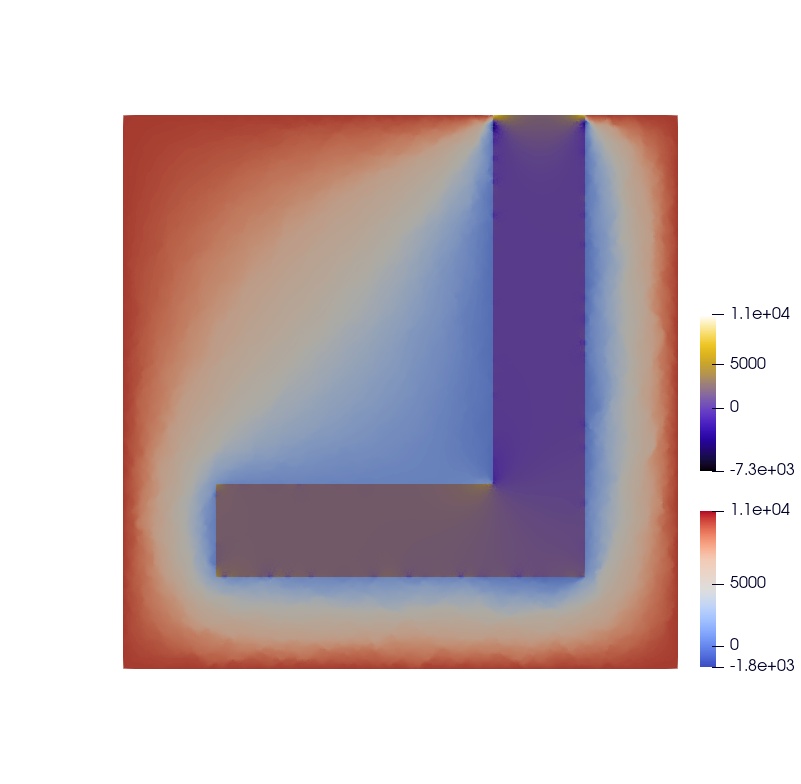}}
  \quad
  \subfloat[Pressure $p_h^m$ in the matrix in
  $\Omega^d$. \label{fig:pressure in matrix3 h well}]
  {\includegraphics[width=0.48\textwidth]{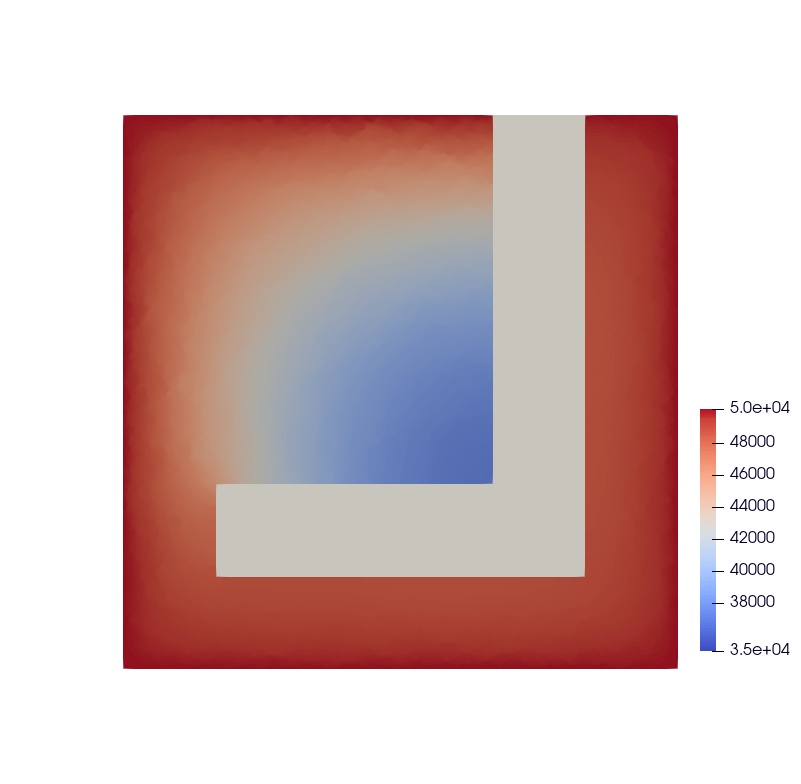}}
  \\
  \subfloat[Velocity field $u_h$ in $\Omega$. \label{fig:velocity3h
    wellbore}]{\includegraphics[width=0.48\textwidth]{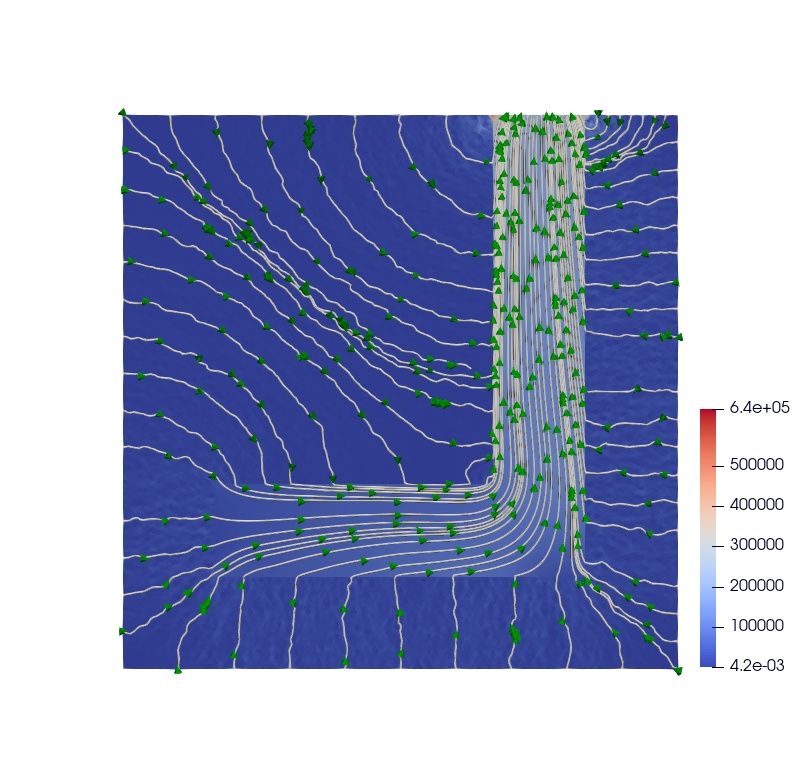}}
  \quad
  \subfloat[Velocity field $u_h^m$ in the matrix
  $\Omega^d$. \label{fig:velocity in
    matrix3}]{\includegraphics[width=0.48\textwidth]{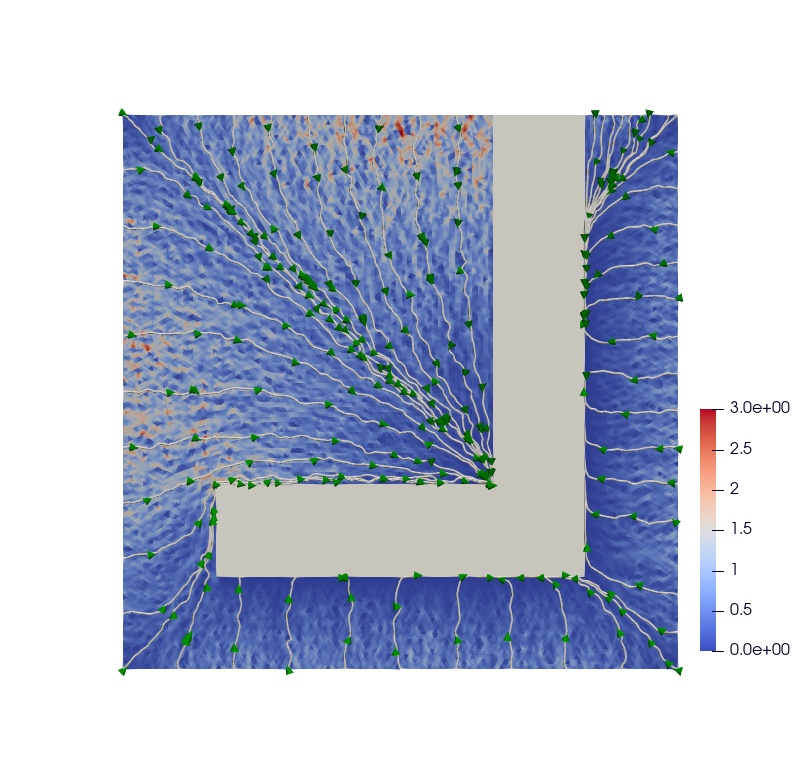}}
  \label{fig:Example3 horizontal wellbore figure}
\end{figure}

%---------------------------------------------------------------------
\section{Conclusions}
\label{sec:Conclusions}

In this paper, we presented a strongly conservative HDG method for the
dual-porosity-Stokes problem. We proved that the discrete problem is
well-posed and presented an a priori error analysis showing optimal
rates of convergence in the energy norm. Our theoretical findings are
supported with numerical examples.

%---------------------------------------------------------------------
\section*{Acknowledgements}

Aycil Cesmelioglu and Jeonghun J. Lee gratefully acknowledge support
by the National Science Foundation (grant numbers DMS-2110782 and
DMS-2110781) and Sander Rhebergen gratefully acknowledges support from
the Natural Sciences and Engineering Research Council of Canada
through the Discovery Grant program  (RGPIN-2023-03237).

%---------------------------------------------------------------------
\bibliographystyle{ieeetr}
\bibliography{references}
%---------------------------------------------------------------------
\end{document}